\documentclass[a4paper,11pt, reqno]{amsart}
\usepackage[utf8]{inputenc}
\usepackage{amsmath}
\usepackage{amssymb}
\usepackage{amsthm}
\usepackage{amsrefs} 
\usepackage{a4wide}
\usepackage{paralist}
\usepackage{marvosym}
\usepackage{esint}
\usepackage{dsfont}
\usepackage{extarrows}
\usepackage{mathrsfs}
\usepackage{pdfpages}
\usepackage{blindtext}

\newtheorem{theorem}{Theorem}[section]

\newtheorem{proposition}{Proposition}[section]

\theoremstyle{definition}

\theoremstyle{remark}
\newtheorem{remark}{Remark}[section]

\DeclareMathOperator{\E}{\mathds{E}}

\DeclareMathOperator{\Var}{Var}
\DeclareMathOperator{\N}{\mathbb{N}}
\DeclareMathOperator{\R}{\mathbb{R}}

\DeclareMathOperator{\1}{\mathds{1}}

\def\sT{\mathscr{T}}

\numberwithin{equation}{section}

\newcommand{\lnorm}[3]{\left\Vert{#3}\right\Vert_{\ell^{#1}(\N)^{\otimes{#2}}}}
\newcommand{\ellnorm}[3]{\left\Vert{#3}\right\Vert_{\ell^{#1}(\N)^{\otimes{#2}}}}
\newcommand{\lnormb}[3]{\big\Vert{#3}\big\Vert_{\ell^{#1}(\N)^{\otimes{#2}}}}

\title[Discrete Malliavin-Stein method and applications]{Discrete Malliavin-Stein method: Berry-Esseen bounds for random graphs and percolation}

\author[K. Krokowski]{Kai Krokowski}
\address{Kai Krokowski: Faculty of Mathematics, NA 3/28, Ruhr University Bochum, Germany}
\email{kai.krokowski@rub.de}

\author[A. Reichenbachs]{Anselm Reichenbachs}
\address{Anselm Reichenbachs: Faculty of Mathematics, NA 3/28, Ruhr University Bochum, Germany}
\email{anselm.reichenbachs@rub.de}

\author[C. Th\"ale]{Christoph Thäle}
\address{Christoph Th\"ale: Faculty of Mathematics, NA 3/68, Ruhr University Bochum, Germany}
\email{christoph.thaele@rub.de}

\date{}

\begin{document}

\begin{abstract}
A new Berry-Esseen bound for non-linear functionals of non-symmetric and non-homogeneous infinite Rade\-macher sequences is established. It is based on a discrete version of the Malliavin-Stein method and an analysis of the discrete Ornstein-Uhlenbeck semigroup. The result is applied to sub-graph counts and to the number of vertices having a prescribed degree in the Erd\H{o}s-Renyi random graph. A further application deals with a percolation problem on trees.
\end{abstract}

\maketitle

\section{Introduction}

The Malliavin-Stein method has become a versatile device for proving quantitative limit theorems. It combines the Malliavin calculus of variations with Stein's method. The results obtained this way typically fall into two categories. The first category consists of limit theorems for non-linear functionals defined on the Wiener space with notable applications to Gaussian random processes, especially the fractional Brownian motion \cite{MarPec,NouPecPTRF09,NouPecBook}, random matrices \cite{NouPecMatrices} and random polynomials \cite{AzaisLeonCLTRandTrigPol}. The other brand comprises limit theorems for functionals of Poisson random measures and their applications to stochastic geometry \cite{EicTha,LacPec1,LastPeccatiSchulte,LPST,PecTha,SchulteThaeleScalingLimits}, $U$-statistics \cite{DecreusefondSchulteThaele,EicTha,LacPec2,PecTha,ReiSch}, non-linear statistics of spherical Poisson fields \cite{DurastaniEtAl1} and the theory of L\'evy processes \cite{EicTha,LastPeccatiSchulte,PecSolTaqUtz}.

On the other hand, the Malliavin-Stein method has left only few traces in that part of probability theory in which discrete random structures are investigated. One exception is the paper \cite{ReiPec}, where Stein's method for normal approximation has been combined with tools from discrete stochastic analysis for symmetric Rademacher sequences to deduce quantitative central limit theorems with respect to probability distances based on smooth test functions. Here, by a symmetric Rademacher sequence we understand an infinite sequence of independent and identically distributed random variables taking the values $\pm 1$ with probability $1/2$ each. This approach has been extended in \cite{KroReiThae} to deduce Berry-Esseen bounds, that is, estimates for the Kolmogorov distance in related central limit theorems. The applications considered in \cite{KroReiThae,ReiPec} concern the number of two-runs, a quantitative version of a combinatorial central limit theorem as well as traces of powers of random Bernoulli matrices. While the previously mentioned papers were concerned with the symmetric case, we work with general non-linear functionals of non-symmetric and even non-homogeneous Rademacher sequences in order to bring a rich class of examples, that were not accessible before, within the reach of the Malliavin-Stein method. Moreover, we emphasize that some of the examples we present below are not within the reach of any of the traditional approaches using Stein's method.

One of the main tools of the Malliavin-Stein method on the Wiener or the Poisson space is the so-called multiplication formula for multiple stochastic integrals, cf.\ \cite{PecTaq} for a general overview. The main difficulty in the discrete set-up is that no such multiplication formula for discrete multiple stochastic integrals based on non-symmetric or non-homogeneous Rademacher sequences is available. Consequently, a new type of abstract Berry-Esseen bound needs to be developed, which is getting along without this technical device. Such a result, namely Theorem \ref{Folgetheorem} below, is one of our main contributions. It can be interpreted as a kind of `second-order Poincar\'e inequality' and is the discrete analogue of corresponding results on the Wiener or the Poisson space, cf.\ \cite{LastPeccatiSchulte,NourdinPeccatiReinertSecondOrderPoincare}. It relies on a generalization of the Malliavin-Stein bound established in \cite{KroReiThae} and on an analysis of the discrete Ornstein-Uhlenbeck semigroup. To make this approach work, we also have to develop further some facets of the discrete Malliavin calculus of variations. In particular, we present a generalization of the integration-by-parts formula, which is one of our crucial devices. The Berry-Esseen bound we obtain this way is particularly well suited for the study of discrete random structures. This is due to the fact that the chaotic decomposition of the functional at hand does not have to be specified. Instead, the impact of local perturbations on the functional measured by means of a certain difference operator (discrete Malliavin derivative) has to be evaluated. A sufficient condition for asymptotic normality is that moments of first- and second-order discrete Malliavin derivatives of the functional are sufficiently small.

\medskip

To highlight the versatility of our general limit theorem we now present a couple of concrete applications. The first one deals with the \textit{triangle counting statistic} associated with the \textit{Erd\H{o}s-Renyi random graph}. Introduced in \cite{ErdosRenyRandomGraphs}, the model has since then become one of the most popular models in discrete probability, cf.\ \cite{JanLucRuc} for an exhaustive list of references. Informally, the random graph $G(n,p)$ is a graph on $n\in\N$ vertices in which each edge between two vertices is included with probability $p\in[0,1]$, independently of the other edges (for a detailed construction see Section \ref{sec:RandomGraphs} below and see Figure \ref{fig:ER} for simulations). In what follows we allow $p$ also to depend on $n$, but for practical reasons we suppress this in our notation. The random variable in the focus of our attention is the number $T=T(n,p)$ of triangles in $G(n,p)$, i.e., the number of sub-graphs of $G(n,p)$ that are isomorphic to the complete graph on $3$ vertices. A comprehensive central limit theorem for the normalized random variable $F:=(T-\E[T])/\sqrt{\Var[T]}$ has been derived in \cite{Ruc} by the method of moments. In particular, it provides a necessary and sufficient condition on $n$ and $p$, which ensures asymptotic Gaussianity for $F$. Namely, as $n\to\infty$, one has that
\begin{align*}
F\overset{d}{\longrightarrow}N\quad\text{if and only if}\quad np\to\infty\text{ and }n^2(1-p)\to\infty\,,
\end{align*}
where $N\sim\mathcal{N}(0,1)$ is a standard Gaussian random variable and $\overset{d}{\longrightarrow}$ indicates convergence in distribution. Using Stein's method for normal approximation, a rate of convergence in this central limit theorem measured by some sort of bounded Wasserstein distance has been established in \cite{BarKarRuc}. If $p\in(0,1)$ is fixed,
$$d_1(F,N):=\sup_{h\in\mathcal{H}}{\big|\E[h(F)]-\E[h(N)]\big|\over\|h\|_\infty+\|h'\|_\infty}=\mathcal{O}(n^{-1})\,,$$
where $\mathcal{H}$ is the class of bounded functions $h:\R\to\R$ with bounded first derivative and where $\|\,\cdot\,\|_\infty$ denotes the supremum norm. For the case that $p=\theta n^{-\alpha}$ with $\alpha\in(0,1)$ and $\theta\in(0,n^{\alpha})$ such that  $\theta\asymp 1$ the result in \cite{BarKarRuc} delivers the bound
\begin{equation*}
d_1(F,N)=\begin{cases}\mathcal{O}\big(n^{-1\,+\,\alpha/2}\big)&\text{if }0<\alpha\leq\frac{1}{2}\\ \mathcal{O}\big(n^{-3(1-\alpha)/2}\big)&\text{if }\frac{1}{2}<\alpha<1\,.\end{cases}
\end{equation*}
We use the following standard notation for comparing the order of magnitude of two real sequences: We write $a_n\asymp b_n$ for two real sequences $(a_n)_{n\in\N}$ and $(b_n)_{n\in\N}$ whenever \[c\leq\liminf\limits_{n\to\infty}\Big|\frac{a_n}{b_n}\Big|\leq\limsup\limits_{n\to\infty}\Big|\frac{a_n}{b_n}\Big|\leq C\]
for two constants $0<c\leq C<\infty$. We also write $a_n=\mathcal{O}(b_n)$ for two non-negative sequences $(a_n)_{n\in\N}$ and $(b_n)_{n\in\N}$ if there is a constant $c\in(0,\infty)$ such that $a_n\leq c\,b_n$ for sufficiently large $n$. Applying a standard smoothing argument, one can show that the more prominent and more natural Kolmogorov distance $$d_K(F,N):=\sup_{x\in\R}\big|P(F\leq x)-P(N\leq x)\big|$$
between $F$ and $N$ is bounded by a constant multiple of the \textit{square-root} of $d_1(F,N)$, cf.\ \cite[Proposition 2.4]{LaplaceDistribution}. However, this typically leads to a suboptimal rate of convergence for the Kolmogorov distance $d_K(F,N)$. For example, in the special case of a fixed $p\in(0,1)$ one expects that also $d_K(F,N)$ is of order $n^{-1}$. Our main contribution in this context is the following Berry-Esseen bound, which in particular confirms that this is in fact true. We emphasize that we are not aware of any other technique, which could be used to provide bounds on the Kolmogorov distance of this quality if $p$ is of the form $\theta n^{-\alpha}$ with $\alpha\in(0,1)$ and $\theta\in(0,n^{\alpha})$ such that  $\theta\asymp 1$. In what follows, we treat both set-ups simultaneously and contribute thereby to a long standing problem in this area. 

\begin{theorem}\label{thm:ERGraph}
Denote by $N\sim\mathcal{N}(0,1)$ a standard Gaussian random variable. Let $p=\theta\,n^{-\alpha}$ with $\alpha\in[0,1)$ and $\theta=\theta_n\in(0,n^{\alpha})$ such that  $\theta\asymp 1$. Then
\begin{equation*}
d_K(F,N)=\begin{cases}\mathcal{O}(n^{-1\,+\,\alpha})&\text{if }0\leq\alpha\leq\frac 12\\ \mathcal{O}(n^{-3/4\,+\,\alpha/2})&\text{if }\frac 12<\alpha\leq\frac 23\\ \mathcal{O}(n^{-5(1-\alpha)/4})&\text{if }\frac 23<\alpha<1\,.\end{cases}
\end{equation*}
In particular, if $p$ is constant, i.e., if $\alpha=0$, $$d_K(F,N)=\mathcal{O}(n^{-1})\,.$$
\end{theorem}


\begin{figure}[t]
\includegraphics[trim = 30mm 30mm 20mm 20mm,clip,width=0.45\columnwidth]{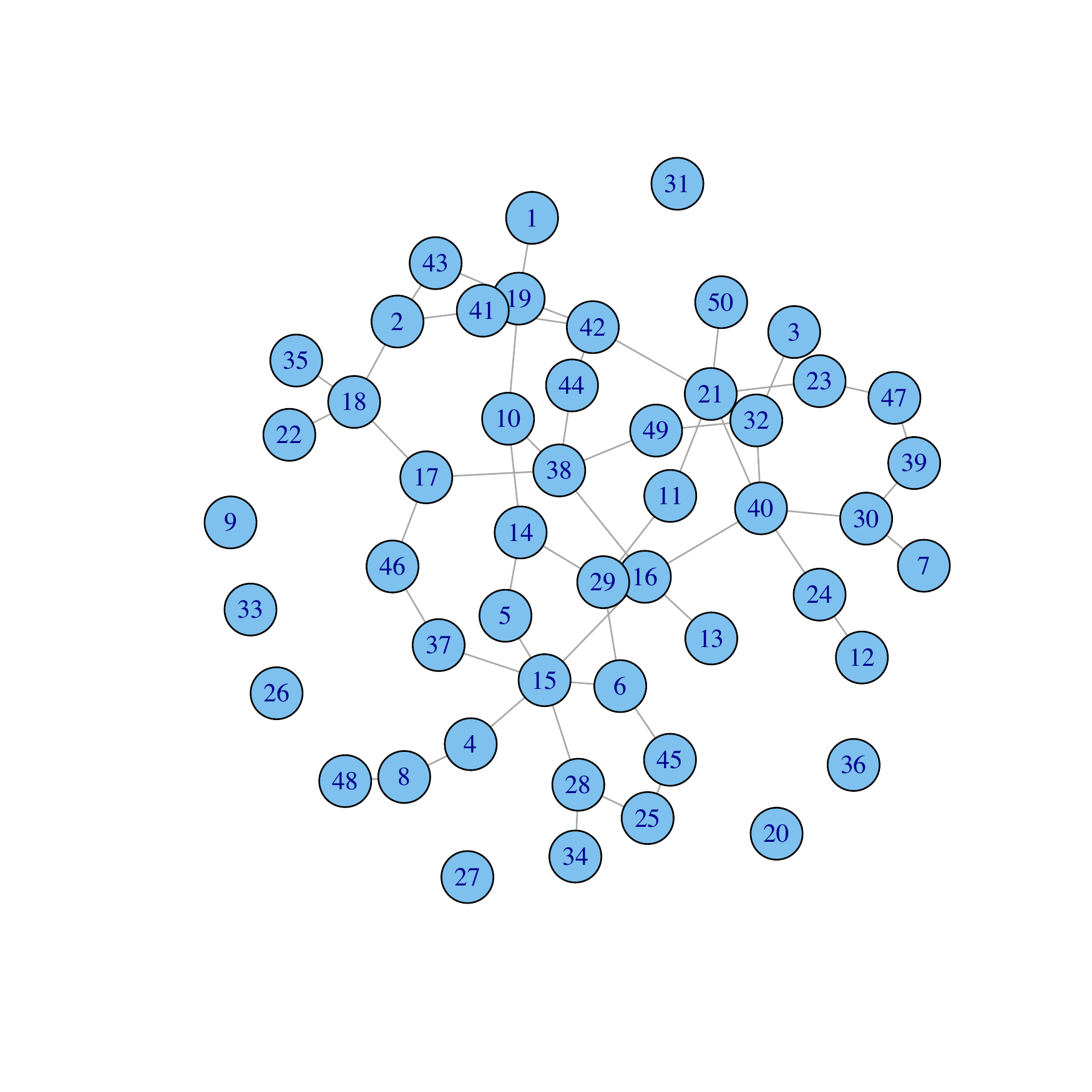}\qquad
\includegraphics[trim = 30mm 30mm 20mm 20mm,clip,width=0.45\columnwidth]{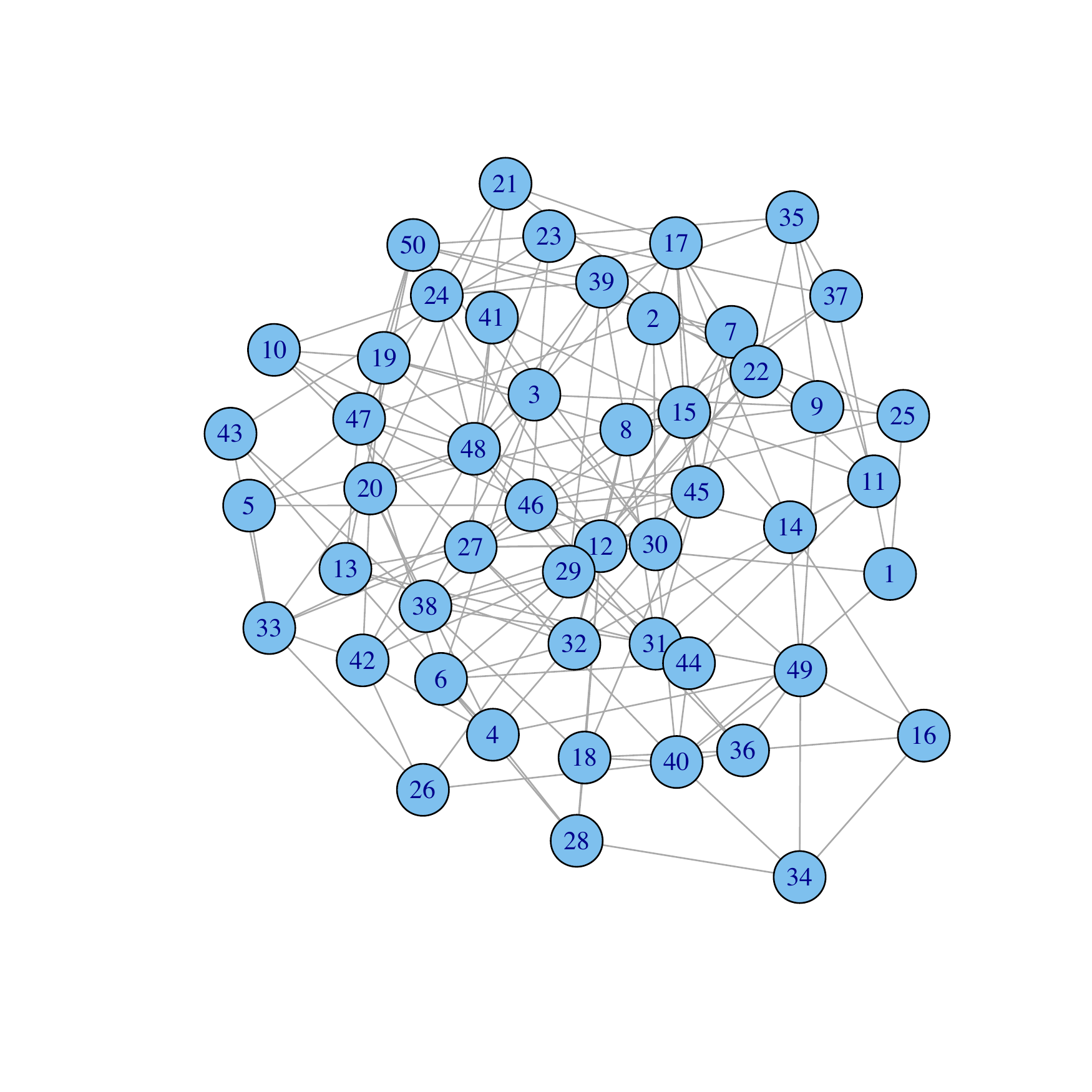}
\caption{Realizations of Erd\H{o}s-Renyi random graphs with $n=50$ vertices and $p=0.04$ (left) and $p=n^{-1/2}\approx 0.14$ (right). The graphics were produced by means of the freely available \texttt{R}-package \texttt{igraph}.}
\label{fig:ER}
\end{figure}

To underline that not only triangle counts are within the reach of our methods, we now consider the problem of \textit{counting copies of general sub-graphs $\Gamma$ in the Erd\H{o}s-Renyi random graph} $G(n,p)$. Formally, we denote by $S=S(n,p)$ the number of copies of $\Gamma$ in $G(n,p)$ and by $F:=(S-\E[S])/\sqrt{\Var[S]}$ the normalized sub-graph counting statistic. We assume that $\Gamma$ has at least one edge and in contrast to Theorem \ref{thm:ERGraph} we also assume that the success probability $p\in(0,1)$ is fixed and does not depend on $n$. In this situation, Theorem 2 in \cite{BarKarRuc} says that
$$d_1(F,N)=\mathcal{O}(n^{-1})$$
and our abstract Berry-Esseen bound can be used to conclude that the $d_1$-distance can be replaced by the Kolmogorov distance.

\begin{theorem}\label{thm:Subgraphs}
Denote by $N\sim\mathcal{N}(0,1)$ a standard Gaussian random variable and fix $p\in(0,1)$. Then $$d_K(F,N)=\mathcal{O}(n^{-1})$$
for all graphs $\Gamma$ having at least one edge.
\end{theorem}

It should be pointed out that Theorem \ref{thm:Subgraphs} can alternatively be obtained from the sharp cumulant estimate in \cite[Proposition 10.3]{ModPhi} together with \cite[Corollary 2.1]{SaulisStatu} or from the Berry-Esseen bound for decomposable random variables in \cite[Theorem 5.1]{Raic}.
 
 \medskip

Besides the number of triangles or general sub-graphs, there are several other random variables associated with the Erd\H{o}s-Renyi random graph that have found considerable attention in the literature. One statistic that has been object of much study is the \textit{number of vertices having a prescribed degree}. For example, in \cite{Kordecki} a central limit theorem for the number of isolated vertices was given, which for general degree is a result in \cite{JansonNowicki}. A rate of convergence for the $d_1$-distance as introduced above has been obtained in \cite{BarKarRuc}. A technically highly sophisticated version of Stein's method was developed in \cite{GoldsteinVertex} to deduce a corresponding Berry-Esseen bound in case that the success probability is $p=\theta/n$. Using our general Berry-Esseen bound, we are able to present a quick and streamlined proof of an extended version of this quantitative central limit theorem. We denote for $d\in\{0,1,2,\ldots\}$ by $V_{n,d}$ the number of vertices of degree $d$ in $G(n,p)$ in case that the success probability satisfies $p=\theta n^{-\alpha}$ for suitable $\alpha\in\R$ and $\theta\in(0,n^{\alpha})$ such that $\theta\asymp 1$. We finally define the normalized random variable $G_{n,d}:=(V_{n,d}-\mathbb{E}[V_{n,d}])/\sqrt{\Var[V_{n,d}]}$, $n\in\N$.

\begin{theorem}\label{thm:degrees}
Denote by $N\sim\mathcal{N}(0,1)$ a standard Gaussian random variable and fix $d\in\{0,1,2,\ldots\}$. Let $p=\theta n^{-\alpha}$ with $\alpha\in[1,2)$ and $\theta\in(0,n^{\alpha})$ such that  $\theta\asymp 1$. Then 
$$d_K(G_{n,d},N)=\begin{cases}\mathcal{O}(n^{-1+\alpha/2}) & \text{if }d=0\,,\alpha\in[1,2)\\ \mathcal{O}(n^{1/2-3d/2-\alpha+3\alpha d/2}) & \text{if }d\in\N\,,\alpha\in[1,{3d-1\over 3d-2})\,.\end{cases}$$
In particular, if $\alpha=1$,
$$d_K(G_{n,d},N)=\mathcal{O}(n^{-1/2})$$
for all $d\in\{0,1,2,\ldots\}$.
\end{theorem}

Our final application deals with the \textit{number of connected components} arising from \textit{bond percolation on a tree}. We recall that a rooted tree $\sT$ is an undirected graph with one distinguished vertex, the root of the tree, in which any two vertices are connected by a unique self-avoiding path. We denote for $n\in\N$ by $\sT_n$ the sub-tree of $\sT$, which consists of all vertices of $\sT$ that have graph-distance at most $n$ from the root. By $|\sT_n|$ we denote the number of edges of $\sT_n$. In what follows, we assume that each vertex of $\sT$ has degree bounded by $D+1$ with $D\in\N$ and that $\sT$ has infinitely many vertices. If the degree of the root is $D$ and if the degree of each other vertex of $\sT$ is $D+1$ for some fixed $D\in\N$, we say that $\sT$ is a $D$-regular tree. Fix $p\in(0,1)$ and assign to each edge $e$ of $\sT$, independently of the other edges, a Rademacher random variable $X_e$ such that $P(X_e=+1)=p$ and $P(X_e=-1)=1-p$. We now remove from $\sT$ all edges $e$ for which $X_e=-1$ and indicate by $\sT(p)$ the resulting random graph, see Figure \ref{fig:tree} for a simulation. Its restriction to $\sT_n$ is denoted by $\sT_n(p)$ and we let $C_n(p)$ be the number of connected components of $\sT_n(p)$. Here, by a connected component we understand a maximal connected sub-graph of $\sT_n(p)$ consisting of at least one edge; isolated vertices are not counted. Our next result is a Berry-Esseen bound for the normalized random variable $H_n(p):=(C_n(p)-\E[C_n(p)])/\sqrt{\Var[C_n(p)]}$. This adds to the qualitative central limit theorem in \cite{SugimineTakei}.

\begin{theorem}\label{thm:PercolationTree}
Fix $p\in(0,1)$ and denote by $N\sim\mathcal{N}(0,1)$ a standard Gaussian random variable. Then $$d_K(H_n(p),N)=\mathcal{O}(|\sT_n|^{-1/2})\,.$$ In particular, in case of a $D$-regular tree one has that
$$d_K(H_n(p),N)=\begin{cases} \mathcal{O}(n^{-1/2}) &\text{if }D=1\\ \mathcal{O}(D^{-n/2}) &\text{if }D\geq 2\,.\end{cases}$$
\end{theorem}

\begin{figure}[t]
\includegraphics[trim = 30mm 30mm 20mm 20mm,clip,width=0.45\columnwidth]{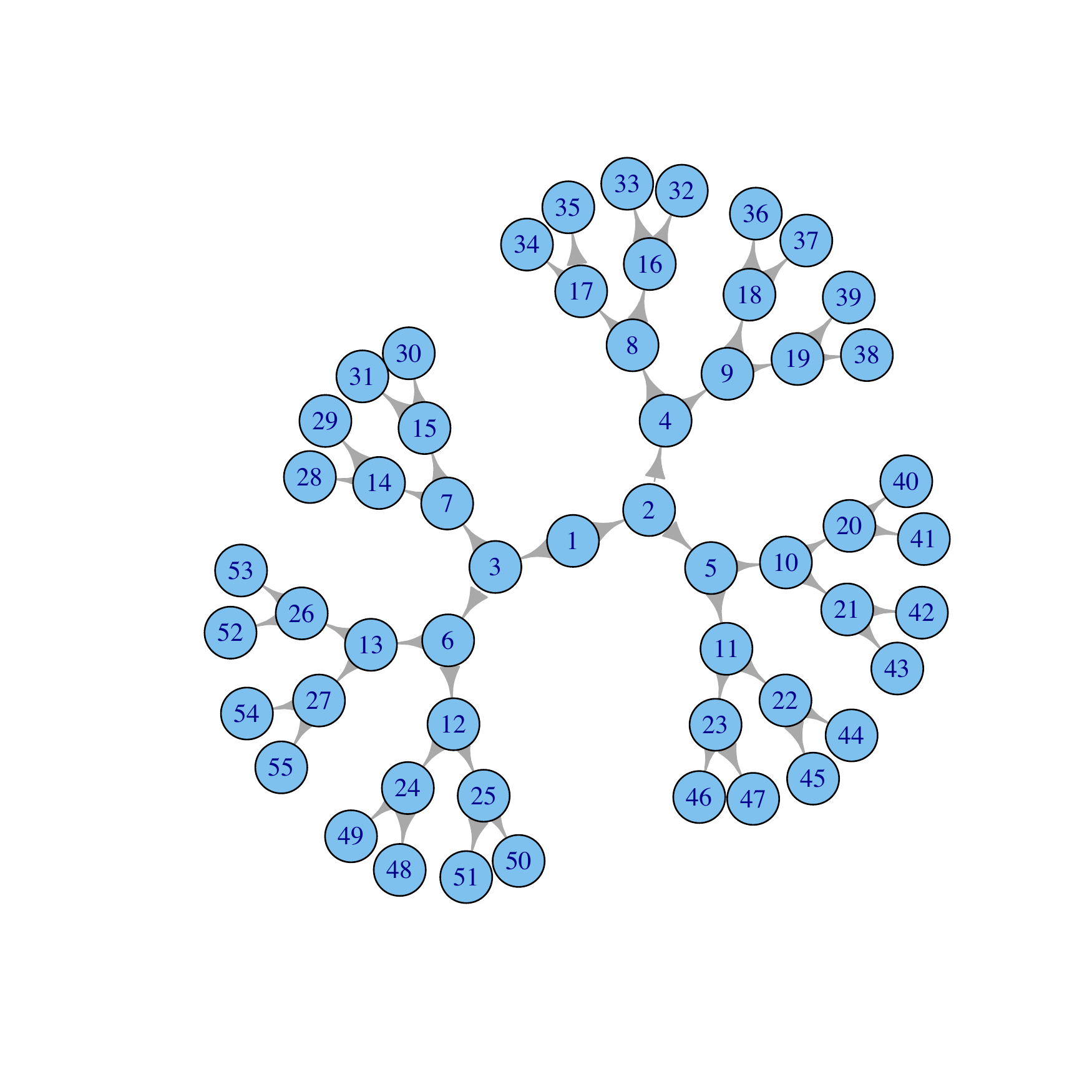}\qquad
\includegraphics[trim = 30mm 30mm 20mm 20mm,clip,width=0.45\columnwidth]{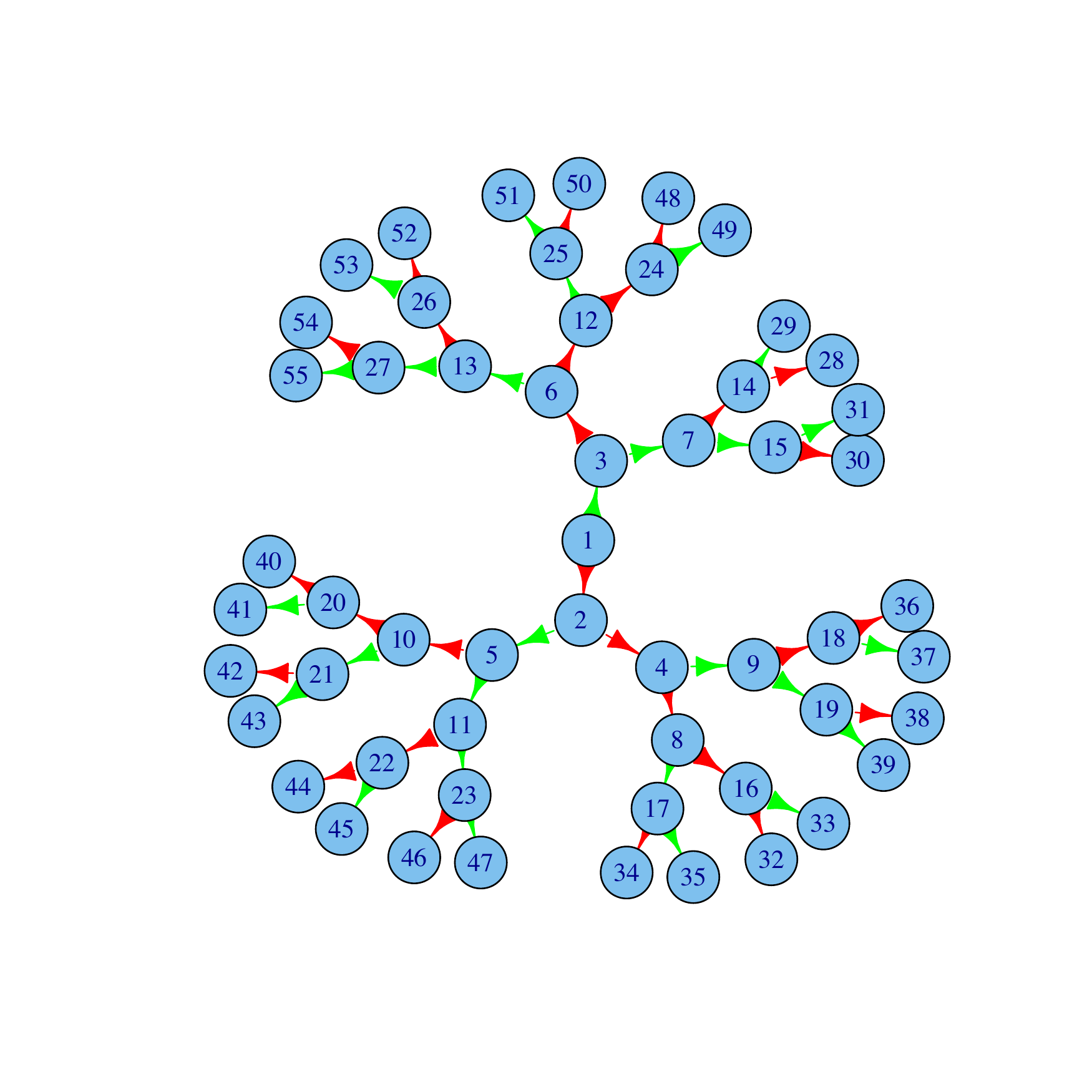}
\caption{$\sT_5$ of a $2$-regular tree $\sT$ (left). Realization of $\sT_5(p)$ with $p=1/2$ (right). The colour red means that the edge is included, while green indicates that the edge has been removed. The graphics were produced by means of the freely available \texttt{R}-package \texttt{igraph}.}
\label{fig:tree}
\end{figure}

The rest of this paper is structured as follows. In Section \ref{sec:Preliminaries} we collect some background material related to the discrete Malliavin calculus. An analysis of the discrete Ornstein-Uhlenbeck semigroup is the content of Section \ref{sec:MehlerFormel}. This is used in Section \ref{sec:AbstractBerryEsseen} to derive our abstract Berry-Esseen bound, which in turn is applied in Section \ref{sec:RandomGraphs} to the Erd\H{o}s-Renyi random graph and in Section \ref{sec:Percolation} to the percolation problem on trees. These sections also contain the proofs of Theorem \ref{thm:ERGraph}, Theorem \ref{thm:Subgraphs}, Theorem \ref{thm:degrees} and Theorem \ref{thm:PercolationTree} presented above.

\bigskip

\textit{Note added in proof}: After submission of the paper it came to our attention that a multiplication formula for discrete multiple stochastic integrals based on non-symmetric and non-homogeneous Rademacher sequences has been developed in a manuscript by Privault and Torrisi that was not available to us, but has now appeared as \cite{PriTor}. See also \cite{Kro}. 

\section{Preliminaries}\label{sec:Preliminaries}

\subsection{Set-up}\label{subsec:SetUp}

For each $k\in\N$ let $0<p_k<1$ and put $q_k:=1-p_k$. We abbreviate the sequences $(p_k)_{k\in\N}$ and $(q_k)_{k\in\N}$ by $p$ and $q$, respectively. By $X := (X_k)_{k\in\N}$ we denote a sequence of independent random variables such that
$$P(X_k=+1)=p_k\qquad\text{and}\qquad P(X_k=-1)=q_k\,,\qquad k\in\N\,.$$
This is what we call a (non-symmetric and non-homogeneous) sequence of independent Rade\-macher random variables. We construct them in the canonical way by taking $(\Omega,\mathcal{F},P)$ as probability space, where $\Omega:=\{-1,+1\}^{\N}$, $\mathcal{F}:=\mathcal{P}(\{ -1,+1 \})^{\otimes \N}$ and $P:=\bigotimes_{k=1}^\infty (p_k\delta_{+1}+q_k\delta_{-1})$, with $\mathcal{P}(M)$ being the power set of a set $M$ and $\delta_{\pm1}$ being the unit-mass Dirac measure concentrated at $\pm1$. We then put $X_k(\omega) := \omega_k$ for each $k \in \N$ and $\omega := (\omega_k)_{k \in \N} \in \Omega$. Note that $X_k$ has mean $p_k-q_k$ and variance $4p_kq_k$.

\subsection{Discrete multiple stochastic integrals}

Denote by $\kappa$ the counting measure on $\N$ and put $\ell^2(\N)^{\otimes n}:=L^2(\N^n,\mathcal{P}(\N)^{\otimes n},\kappa^{\otimes n})$ for $n\in\N$. In the following, we refer to the elements of $\ell^2(\N)^{\otimes n}$ as kernels. By $\ell^2(\N)^{\circ n}$ we denote the class of symmetric kernels and $\ell_0^2(\N)^{\circ n}$ stands for the sub-class of symmetric kernels vanishing on diagonals, i.e., vanishing on the complement of the set $\Delta_n:=\{(i_1,\dots,i_n)\in\N^n: i_k\neq i_\ell \text{ for } k\neq \ell \}$. We further put $\ell^2(\N)^{\otimes 0}:=\R$.

For $n\in\N$ and a kernel $f\in\ell_0^2(\N)^{\circ n}$ we define the discrete multiple stochastic integral of order $n$ of $f$ as
$$J_n(f):=n!\sum_{1\leq i_1<\ldots<i_n<\infty}f(i_1,\ldots,i_n)\,Y_{i_1}\cdots Y_{i_n}\,,$$
where $(Y_k)_{k\in\N}$ with $Y_k := (X_k - p_k + q_k) / (2\sqrt{p_k q_k})$ stands for the normalized sequence of independent Rademacher random variables as introduced above. We also put $J_0(c)=c$ for $c\in\R$. The space spanned by the random variables of the form $J_n(f)$ with $f\in\ell_0^2(\N)^{\circ n}$ is called the Rademacher chaos of order $n$.

Discrete multiple stochastic integrals of different orders are mutually orthogonal and satisfy the isometry relation
\begin{align}\label{Isometry formula}
\E[J_{n}(f)J_{m}(g)]=\1_{\{n=m\}} n! \langle f,g \rangle_{\ell^2(\N)^{\otimes n}}
\end{align}
for all $n,m \in \N$ and kernels $f\in\ell_0^2(\N)^{\circ n}$, $g\in\ell_0^2(\N)^{\circ m}$. Moreover, it is a classical fact that every $F\in L^2(\Omega)$ (i.e., every square-integrable Rademacher functional) admits a chaotic decomposition
\begin{align}\label{Chaos representation}
F=\E[F]+\sum_{n=1}^\infty J_n(f_n)
\end{align}
for uniquely determined kernels $f_n\in\ell_0^2(\N)^{\circ n}$, where the series converges in $L^2(\Omega)$, cf.\ \cite[Proposition 6.7]{Pri}. Together with the isometry relation for discrete multiple stochastic integrals this decomposition implies that the variance of $F$ is given by
$$\Var[F]=\sum_{n=1}^\infty n!\|f_n\|^2_{\ell^2(\N)^{\otimes n}}\,.$$

\subsection{Malliavin calculus}

In this section we introduce some basic notions from discrete Malliavin calculus and refer to \cite{Pri} for further details and background material. Let $F\in L^2(\Omega)$. The discrete gradient of $F$ in direction $k\in\N$ is given by
\begin{equation}\label{eq:PathWiseDifferenceOperator}
D_kF:=\sqrt{p_kq_k}\,(F_k^+-F_k^-)\,,
\end{equation}
where $F_k^{\pm}(\omega):=F(\omega_1, \ldots, \omega_{k-1}, \pm 1, \omega_{k+1},\ldots)$. Note that the normalization in \eqref{eq:PathWiseDifferenceOperator} is chosen such that $D_kY_k=1$.

The discrete gradient satisfies the following product formula. Namely, if $F,G\in L^2(\Omega)$ and $k\in\N$, then
\begin{equation}\label{eq:ProduktFormel}
D_k(FG)  = (D_kF)G+F(D_kG)-{X_k\over\sqrt{p_kq_k}}(D_kF)(D_kG)\,,
\end{equation}
see \cite[Proposition 7.8]{Pri}. We remark that in contrast to classical Malliavin calculus (see \cite{NualartBook}), the product formula in the discrete set-up carries the additional term $$-({X_k/\sqrt{p_kq_k}})(D_kF)(D_kG)\,,$$ which is not present in the continuous framework. A similar effect also happens on the Poisson space, cf.\ \cite{PecSolTaqUtz} and the references cited therein.

The iterated discrete gradient $D^nF := (D_{k_1,\ldots,k_n}^n F)_{k_1,\ldots,k_n\in\N}$ of order $n\geq 2$ is defined by $D_{k_1,\ldots,k_n}^nF:=D_{k_n}(D_{k_1,\ldots,k_{n-1}}^{n-1}F)$ for $k_1,\ldots,k_n\in\N$, where we put $D_k^1F:=D_kF$.

We now present a formula which allows to compute the kernels $f_n$ in a chaotic decomposition as in \eqref{Chaos representation}. In the framework of classical Malliavin calculus this is known as Stroock's formula. Since we have not found such a result for general Rademacher functionals in the literature, we provide the detailed arguments (for the special symmetric case see Lemma 2.2 in \cite{KroReiThae} and Section 2.4 in \cite{ReiPec}).

\begin{proposition}[Stroock's formula]\label{Stroock lemma}
Assume that $F \in L^2(\Omega)$ has chaotic decomposition $F=\E[F]+\sum_{n=1}^\infty J_n(f_n)$. Then for every $n \in \N$ it holds that
\begin{equation}
\E[D_{k_1, \dotsc, k_n}^nF] = \E[F\,Y_{k_1} \cdots Y_{k_n}]\,,\quad(k_1, \dotsc, k_n) \in\Delta_n\,, \label{Stroock}
\end{equation}
and 
\begin{equation}
\E[D_{k_1, \dotsc, k_n}^nF]= n!f_n(k_1, \dotsc, k_n)\,\label{Stroock kernels},\quad(k_1, \dotsc, k_n) \in\N^n.
\end{equation}
\end{proposition}
\begin{proof}
We start by proving \eqref{Stroock} by induction. Choosing $G=Y_k$ with $D_kG=D_kY_k=1$ in \eqref{eq:ProduktFormel} yields
\begin{align*}
D_k(FY_k) = (D_kF)Y_k + F - \frac{X_k}{\sqrt{p_kq_k}}D_kF
\end{align*}
and hence
\begin{align}\label{Stroock proof equation 1}
D_k(FY_k)Y_k &= (D_kF)Y_k^2 + FY_k - \frac{X_kY_k}{\sqrt{p_kq_k}}D_kF\,.
\end{align}
It immediately follows from \eqref{eq:PathWiseDifferenceOperator} that, for every $F \in L^2(\Omega)$, $D_kF$ is independent of $X_k$. Therefore, by taking expectations on both sides of \eqref{Stroock proof equation 1} and computing $\E[X_kY_k]=2\sqrt{p_kq_k}$, we get
\begin{align*}
0 &= \E[D_kF] + \E[FY_k] - 2\E[D_kF]= \E[FY_k] - \E[D_kF]\,,
\end{align*}
which proves \eqref{Stroock} for $n=1$. Now, assume that \eqref{Stroock} holds for some fixed $n \in \N$ and consider
\begin{align*}
\E[D_{k_1, \dotsc, k_{n+1}}^{n+1}F] = \E[D_{k_{n+1}}(D_{k_1, \dotsc, k_{n}}^{n}F)]\,.
\end{align*}
From the case $n=1$ treated above it follows that
\begin{align*}
\E[D_{k_1, \dotsc, k_{n+1}}^{n+1}F] = \E[(D_{k_1, \dotsc, k_{n}}^{n}F)Y_{k_{n+1}}]
\end{align*}
and since $Y_{k_{n+1}}$ behaves like a constant from the point of view of $D_{k_1, \dotsc, k_{n}}^{n}$, our assumption leads to
\begin{align*}
\E[D_{k_1, \dotsc, k_{n+1}}^{n+1}F] &= \E[D_{k_1, \dotsc, k_{n}}^{n}(FY_{k_{n+1}})]= \E[FY_{k_1} \cdots Y_{k_{n+1}}]\,,
\end{align*}
which concludes the proof of \eqref{Stroock}. Identity \eqref{Stroock kernels} then immediately follows from \eqref{Chaos representation} for $(k_1, \dotsc, k_n) \in\Delta_n$. For  $(k_1, \dotsc, k_n) \in\Delta_n^c$ both sides of \eqref{Stroock kernels} are equal to zero.
\end{proof}

For the rest of this section, let $F \in L^2(\Omega)$ have chaotic decomposition $F=\E[F]+\sum_{n=1}^\infty J_n(f_n)$ with kernels $f_n \in \ell_0^2(\N)^{\circ n}$ for $n \in \N$. Since $D_kF \in L^2(\Omega)$ for every $k \in \N$, the discrete gradient also has a chaotic decomposition. Note that the kernels of this decomposition can be deduced from the chaotic decomposition of $F$ using Stroock's formula. More precisely, the $n$'th kernel of the chaotic decomposition of $D_kF$ evaluated at $(k_1,\ldots,k_n)\in\N^n$ is given by
\begin{align*}
\frac{1}{n!} \E[D_{k_1, \dotsc, k_n}^n(D_kF)] = \frac{1}{n!} \E[D_{k_1, \dotsc, k_n, k}^{n+1}F] = (n+1)f_{n+1}(k_1, \dotsc, k_n, k)\,.
\end{align*}
Thus, the discrete gradient can be written as
\begin{equation}\label{eq:DefGradient}
D_kF=\sum_{n=1}^{\infty}nJ_{n-1}(f_n(\,\cdot\,, k))\,,
\end{equation}
where $f_n( \, \cdot \,, k) \in \ell_0^2(\N)^{\circ n-1}$ denotes the kernel $f_n$ with one of its components fixed, thus acting as function in $n-1$ variables. For $F \in L^2(\Omega)$ as above and $m \in \N$, we say that $F \in {\rm dom}(D^m)$, if
\begin{align*}
\E[\|D^mF\|^2_{\ell^2(\N)^{\otimes m}}]=\sum_{n=m}^\infty \left( \frac{n!}{(n-m)!} \right)^2 (n-m)!\|f_n\|^2_{\ell^2(\N)^{\otimes n}} < \infty \,.
\end{align*}

Next, we define the Ornstein-Uhlenbeck operator $L$ and its (pseudo-)inverse $L^{-1}$. The domain of $L$ is the class of all  $F\in L^2(\Omega)$ for which $$\E[(LF)^2]=\sum_{n=1}^\infty n^2n!\|f_n\|^2_{\ell^2(\N)^{\otimes n}} < \infty \,.$$
For $F\in{\rm dom}(L)$ we put
$$LF:=-\sum_{n=1}^{\infty}nJ_n(f_n)\,.$$
The discrete Ornstein-Uhlenbeck semigroup $(P_t)_{t\geq 0}$ associated with $L$ is defined as
\begin{equation}\label{eq:OUSemigroup}
P_tF:=\sum_{n=0}^{\infty}e^{-nt}J_n(f_n)\,,\qquad t\geq 0\,.
\end{equation}
The properties of this semigroup will be discussed in detail in Section \ref{sec:MehlerFormel} below.
Moreover, for centred $F\in L^2(\Omega)$ we put
$$ L^{-1}F:=-\sum_{n=1}^{\infty}{1\over n}J_n(f_n)$$
and call $L^{-1}$ the (pseudo-)inverse of the Ornstein-Uhlenbeck operator $L$.


Furthermore, we introduce the discrete divergence operator $\delta$ and its domain ${\rm dom}(\delta)$.
For $u:=(u_k)_{k \in \N} \in (L^2(\Omega))^{\N}$ with
\begin{align*}
u_k := \sum_{n=0}^\infty J_n(g_{n+1}(\, \cdot \,, k))\,,
\end{align*}
where $g_{n+1} \in \ell_0^2(\N)^{\circ n} \otimes \ell^2(\N)$ for $n \in \N$, we say that $u \in {\rm dom}(\delta)$, if
\begin{align}\label{eq:DomDelta}
\sum_{n=0}^\infty (n+1)! \lnormb{2}{n+1}{\widetilde{g_{n+1}} \1_{\Delta_{n+1}}}^2 < \infty\,.
\end{align}
Here and in the following $\tilde{f}(k_1,\ldots,k_n):=\frac{1}{n!}\sum_{\sigma}f(k_{\sigma(1)},\ldots,k_{\sigma(n)})$ denotes the canonical symmetrization of a function $f$ in $n$ variables, where the sum runs over all permutations $\sigma$ of the set $\{1,\ldots,n\}$.\\
For $u \in {\rm dom}(\delta)$, the discrete divergence operator is defined as
\begin{align*}\label{eq:DefDivergenceOperator}
\delta(u):=\sum_{n=0}^\infty J_{n+1}(\widetilde{g_{n+1}} \1_{\Delta_{n+1}})\,.
\end{align*}
Note that, for $u \in {\rm dom}(\delta)$, \eqref{eq:DomDelta} is equivalent to
\begin{align*}
\E[\delta(u)^2] < \infty\,.
\end{align*}
As the adjoint of the discrete gradient, $\delta$ satisfies the integration-by-parts formula
\begin{equation}\label{eq:IntegrationByPartsOriginal}
\E[F\delta(u)] = \E[\langle DF,u\rangle_{\ell^2(\N)}]
\end{equation}
for $F\in{\rm dom}(D)$ and $u\in{\rm dom}(\delta)$, cf.\ \cite[Proposition 9.2]{Pri}.
The operators $D$, $L$ and $\delta$ are related by the identity 
\begin{equation}\label{eq:-DeltaD=L}
-\delta D = L\,.
\end{equation}
In this paper, we make use of the following crucial consequence of \eqref{eq:IntegrationByPartsOriginal} and \eqref{eq:-DeltaD=L}. If $f:\R\to\R$ is measurable and $F\in L^2(\Omega)$ is centred with $f(F)\in{\rm dom}(D)$, then
\begin{equation}\label{eq:IntegrationByParts}
\E[Ff(F)] = \E[\langle Df(F),-DL^{-1}F\rangle_{\ell^2(\N)}]\,.
\end{equation}
Indeed, using \eqref{eq:IntegrationByPartsOriginal} and \eqref{eq:-DeltaD=L} we have that
\begin{align*}
\E[Ff(F)] = \E[LL^{-1}Ff(F)] = \E[-\delta DL^{-1}Ff(F)] = \E[\langle Df(F),-DL^{-1}F\rangle_{\ell^2(\N)}]\,.
\end{align*}

Now, we present an analogue of the integration-by-parts formula \eqref{eq:IntegrationByPartsOriginal} for functionals $F \in L^2(\Omega)$ that do not necessarily belong to ${\rm dom}(D)$ (we refer to Lemma 2.2 in \cite{LastPeccatiSchulte} for a related result on the Poisson space).

\begin{proposition}\label{Indicator adjointness lemma} Let $F \in L^2(\Omega)$. Furthermore, let $u:=(u_k)_{k \in \N} \in (L^2(\Omega))^{\N}$ with
	\begin{align*}
	u_k := \sum_{n=0}^\infty J_n(g_{n+1}(\, \cdot \,, k))\,,
	\end{align*}
	where $g_{n+1} \in \ell_0^2(\N)^{\circ n} \otimes \ell^2(\N)$ for $n \in \N$ and
\begin{align}\label{Indicator adjointness u}
\sum_{n=0}^\infty (n+1)! \ellnorm{2}{n+1}{g_{n+1}}^2 < \infty\,.
\end{align}
Further assume that $(D_kF)u_k \geq 0$ $P$-almost surely for every $k \in \N$. Then $u\in\text{dom}(\delta)$ and
\begin{align}\label{Indicator adjointness}
\E[F\delta(u)] = \E[\langle DF, u \rangle_{\ell^2(\N)}]\,.
\end{align}
\end{proposition}
\begin{proof}
Note that \eqref{Indicator adjointness u} implies \eqref{eq:DomDelta} and hence $u\in\text{dom}(\delta)$. Since $F \in L^2(\Omega)$, it can be represented as
\begin{align*}
F = \sum_{n=0}^\infty J_n(f_n)
\end{align*}
with kernels $f_0:=\E[F]$ and $f_n \in \ell_0^2(\N)^{\circ n}$ for $n \in \N$. The isometry in \eqref{Isometry formula} yields
\begin{align}\label{Indicator adjointness proof equation 1}
\E[F\delta(u)] &= \E \Big[ \Big( \sum_{n=0}^\infty J_n(f_n) \bigg) \bigg( \sum_{n=0}^\infty J_{n+1}(\widetilde{g_{n+1}} \1_{\Delta_{n+1}}) \Big) \Big] \notag\\
&= \sum_{n=0}^\infty (n+1)! \langle f_{n+1}, \widetilde{g_{n+1}} \1_{\Delta_{n+1}} \rangle_{\ell^2(\N)^{\otimes n+1}} \notag\\
&= \sum_{n=0}^\infty (n+1)! \langle f_{n+1}, g_{n+1} \rangle_{\ell^2(\N)^{\otimes n+1}}\,.
\end{align}
Note that the last step in \eqref{Indicator adjointness proof equation 1} is valid, since, for every $n \in \N$, $f_n$ is symmetric and vanishes on diagonals.

Since $(D_kF)u_k \geq 0$ $P$-almost surely for every $k \in \N$ and by the isometry formula for discrete multiple stochastic integrals, we get
\begin{align}\label{Indicator adjointness proof equation 2}
\E[\langle DF, u \rangle_{\ell^2(\N)}] &= \sum_{k=1}^\infty \E[(D_kF)u_k] \notag\\
&= \sum_{k=1}^\infty \E \Big[ \Big( \sum_{n=0}^\infty (n+1)J_n(f_{n+1}(\, \cdot \, ,k)) \Big) \Big( \sum_{n=0}^\infty J_n(g_{n+1}(\, \cdot \, ,k)) \Big) \Big] \notag\\
&= \sum_{k=1}^\infty \sum_{n=0}^\infty (n+1)! \langle f_{n+1}(\, \cdot \, ,k), g_{n+1}(\, \cdot \, ,k) \rangle_{\ell^2(\N)^{\otimes n}} \notag\\
&= \sum_{n=0}^\infty (n+1)! \sum_{k=1}^\infty \langle f_{n+1}(\, \cdot \, ,k), g_{n+1}(\, \cdot \, ,k) \rangle_{\ell^2(\N)^{\otimes n}} \notag\\
&= \sum_{n=0}^\infty (n+1)! \langle f_{n+1}, g_{n+1} \rangle_{\ell^2(\N)^{\otimes n+1}}\,.
\end{align}
Note that the exchange of summation in the penultimate step of \eqref{Indicator adjointness proof equation 2} is valid by Fubini's theorem, since a repeated application of the Cauchy-Schwarz inequality yields that
\begin{align*}
&\sum_{n=0}^\infty \sum_{k=1}^\infty \big|(n+1)! \langle f_{n+1}(\, \cdot \, ,k), g_{n+1}(\, \cdot \, ,k) \rangle_{\ell^2(\N)^{\otimes n}}\big|\\
&\leq \sum_{n=0}^\infty (n+1)! \sum_{k=1}^\infty \ellnorm{2}{n}{f_{n+1}(\, \cdot \, ,k)} \ellnorm{2}{n}{g_{n+1}(\, \cdot \, ,k)}\\
&\leq \sum_{n=0}^\infty (n+1)! \Big( \sum_{k=1}^\infty \ellnorm{2}{n}{f_{n+1}(\, \cdot \, ,k)}^2 \Big)^\frac{1}{2} \Big( \sum_{k=1}^\infty \ellnorm{2}{n}{g_{n+1}(\, \cdot \, ,k)}^2 \Big)^\frac{1}{2}\\
&= \sum_{n=0}^\infty (n+1)! \ellnorm{2}{n+1}{f_{n+1}} \ellnorm{2}{n+1}{g_{n+1}}\\
&\leq \Big( \sum_{n=0}^\infty (n+1)! \ellnorm{2}{n+1}{f_{n+1}}^2 \Big)^\frac{1}{2} \Big( \sum_{n=0}^\infty (n+1)! \ellnorm{2}{n+1}{g_{n+1}}^2 \Big)^\frac{1}{2}\\
&\leq (\E[F^2])^\frac{1}{2} \Big( \sum_{n=0}^\infty (n+1)! \ellnorm{2}{n+1}{g_{n+1}}^2 \Big)^\frac{1}{2} < \infty\,.
\end{align*}
Comparing \eqref{Indicator adjointness proof equation 1} and \eqref{Indicator adjointness proof equation 2} completes the proof.
\end{proof}

Finally, we recall the following Skorohod isometric formula for the discrete divergence operator. Namely, for all $u\in\text{dom}(\delta)$ it holds that
\begin{equation}\label{eq:SkorohodIsometry}
\E [\delta(u)^2]=\E[\|u\|_{\ell^2(\N)}^2]+\E\Big[\sum_{k,\ell=1}^\infty (D_ku_\ell)(D_\ell u_k)\Big]
\end{equation}
according to Proposition 9.3 in \cite{Pri}.

\section{The discrete Ornstein-Uhlenbeck semigroup}\label{sec:MehlerFormel}

For real $t \geq 0$ define the random sequence $X^t:=(X_k^t)_{k\in\N}$ by
$$X_k^t:=X_k^*\,\1_{\{Z_k\leq t\}}+X_k\,\1_{\{Z_k>t\}}\,,$$ 
where $(X_k^*)_{k\in\N}$ is an independent copy of the Rademacher sequence $X=(X_k)_{k\in\N}$ and $(Z_k)_{k\in\N}$ is a sequence of independent and exponentially distributed random variables with mean $1$, independent of all other random variables.

Our first result is a discrete analogue of Mehler's formula on the Wiener or Poisson chaos for which we refer to \cite{NouPecPTRF09} and \cite{LastPeccatiSchulte}, respectively. It expresses the discrete Ornstein-Uhlenbeck semigroup $(P_t)_{t\geq 0}$ defined at \eqref{eq:OUSemigroup} in terms of a conditional expectation. Note that this has already been shown in \cite[Proposition 10.8]{Pri}. Since Mehler's formula is a central device in our approach, we include an elementary and direct proof. 

\begin{proposition}[Mehler's formula]\label{Mehler}
Let $F\in L^2(\Omega)$. The process $(X^t)_{t \geq 0}$ is the Ornstein-Uhlenbeck process associated with $(P_t)_{t \geq 0}$ by the relation
$$P_tF = \E[F(X^t) \, \vert \, X] \qquad P\text{-a.s.}$$
for all $t \geq 0$.
\end{proposition}

\begin{proof}
We first notice that for each $t \geq 0$, $(X_k^t)_{k \in \N}$ is a sequence of independent Rademacher random variables with the same distribution as the sequence $(X_k)_{k \in \N}$. Thus, if $F = \E[F] + \sum_{n=1}^\infty J_n(f_n)$, then $F(X^t)$ has chaotic decomposition
\begin{align}\label{OUDecomp}
F(X^t) = \E[F] + \sum_{n=1}^\infty n! \sum_{1\leq i_1<\ldots<i_n<\infty}f_n(i_1,\ldots,i_n)\,Y_{i_1}^t \cdots Y_{i_n}^t \,,
\end{align}
where both decompositions share the same kernels $f_n \in \ell_0^2(\N)^{\circ n}$, for $n \in \N$, and where the sequence $(Y_k^t)_{k \in \N}$ with 
\begin{equation}\label{eq:processY}
Y_k^t := (X_k^t - p_k + q_k) / (2\sqrt{p_k q_k})=Y_k^*\,\1_{\{Z_k\leq t\}}+Y_k\,\1_{\{Z_k>t\}},
\end{equation}
 for $t \geq 0$, is the normalization of the sequence $(X_k^t)_{k \in \N}$. Here, the random variable $Y_k^*$ is the normalization of $X_k^*$ for every $k\in\N$. Using the independence of the sequences $(X_k)_{k \in \N}$, $(X_k^*)_{k \in \N}$ and $(Z_k)_{k \in \N}$ we deduce from \eqref{eq:processY} that
\begin{align*}
\E[Y_k^t \, \vert \, X_k] =\E[Y_k^*]\cdot P(Z_k\leq t)+Y_k\cdot P(Z_k>t) = Y_k \, e^{-t} \,.
\end{align*}
For a functional $F_d$ only depending on the first $d$ Rademacher random variables we compute by using the chaotic decomposition in \eqref{OUDecomp} as well as linearity and independence,
\begin{align}\label{FiniteMehler}
&\E[F_d(X_1^t, \dotsc, X_d^t) \, \vert \, X] \notag\\
&= \E[F_d(X_1,\ldots,X_d)] + \sum_{n=1}^d n! \sum_{1 \leq i_1<\ldots<i_n \leq d}f_n^{(d)}(i_1,\ldots,i_n) \, \E[Y_{i_1}^t \, \vert \, X_{i_1}] \cdots \E[Y_{i_n}^t \, \vert \, X_{i_n}] \notag\\
&= \E[F_d(X_1,\ldots,X_d)] + \sum_{n=1}^d e^{-nt} \, n! \sum_{1 \leq i_1<\ldots<i_n \leq d}f_n^{(d)}(i_1,\ldots,i_n) \, Y_{i_1} \cdots Y_{i_n} \notag\\
&= \E[F_d(X_1,\ldots,X_d)] + \sum_{n=1}^d e^{-nt} \, J_n(f_n^{(d)}) \notag\\
&= P_t F_d(X_1,\ldots,X_d) \,.
\end{align}
The general case follows from \eqref{FiniteMehler} due to the fact that the set of functionals depending only on finitely many Rademacher variables is dense in $L^2(\Omega)$ and that  both sides of \eqref{FiniteMehler} are continuous functions of $F_d$. 
\end{proof}

As a next step, we derive an integral representation for the expression $-D^mL^{-1}F$, i.e., the $m$-fold iterated discrete gradient applied to $-L^{-1}F$.

\begin{proposition}\label{IntMehler}
For $m,k_1,\ldots,k_m\in\N$ and centred $F\in{\rm dom}(D^m)$ one has that $$-D_{k_1,\ldots,k_m}^mL^{-1}F=\int_0^\infty e^{-mt}P_tD_{k_1,\ldots,k_m}^mF\, dt\qquad P\text{-a.s.}$$
\end{proposition}

\begin{proof}
Since $F\in L^2(\Omega)$ is centred, there are kernels $f_n \in \ell_0^2(\N)^{\circ n}$, $n\in\N$, such that $F=\sum_{n=1}^{\infty}J_n(f_n)$. Fix $d\in\{m,m+1,\ldots\}$ and consider the truncated functional $F_d:=\sum_{n=1}^{d}J_n(f_n)$. Then,
\begin{align}\label{FiniteIntMehler}
-D_{k_1,\ldots,k_m}^mL^{-1}F_d &= \sum_{n=m}^d \frac{(n-1)!}{(n-m)!} J_{n-m}(f_n(\,\cdot\,,k_1,\ldots,k_m)) \notag\\
& = \int_0^\infty e^{-mt} \sum_{n=m}^d \frac{n!}{(n-m)!} e^{-(n-m)t} J_{n-m}(f_n(\,\cdot\,,k_1,\ldots,k_m)) \,dt\,,
\end{align}
where we used that $\int_0^\infty n e^{-nt}\,dt=1$.
By continuity of $D_{k_1,\ldots,k_m}^m$ and $L^{-1}$ one has that
$-D_{k_1,\ldots,k_m}^mL^{-1}F_d$ converges to $-D_{k_1,\ldots,k_m}^mL^{-1}F$
in $L^2(\Omega)$, as $d\to\infty$. To show that the right hand side of \eqref{FiniteIntMehler} converges to
$$\int_0^\infty e^{-mt}P_tD_{k_1,\ldots,k_m}^mF\, dt$$
in $L^2(\Omega)$, as $d \to \infty$, we consider the remainder term
\begin{align*}
R_{m,d} &:= \int_0^\infty e^{-mt}P_tD_{k_1,\ldots,k_m}^mF\, dt-(-D_{k_1,\ldots,k_m}^mL^{-1}F_d)\\
\end{align*}
and show that $\E[R_{m,d}^2]$ vanishes, as $d\to\infty$. First, use \eqref{FiniteIntMehler} to see that
$$R_{m,d}=\int_0^\infty e^{-mt} \sum_{n=d+1}^\infty {n!\over(n-m)!}e^{-(n-m)t}J_{n-m}(f_n(\,\cdot\,,k_1,\ldots,k_m))\, dt\,.$$
We then apply Jensen's inequality, Fubini's theorem and the isometry property of discrete multiple stochastic integrals to conclude that
\begin{align*}
&\E[R_{m,d}^2]\\
 &= \E\Big[\Big(\int_0^\infty e^{-mt} \sum_{n=d+1}^\infty {n!\over(n-m)!}e^{-(n-m)t}J_{n-m}(f_n(\,\cdot\,,k_1,\ldots,k_m))\, dt\Big)^2\Big]\\
&\leq  \E\Big[\int_0^\infty e^{-(2m-1)t} \Big(\sum_{n=d+1}^\infty {n!\over(n-m)!}e^{-(n-m)t}J_{n-m}(f_n(\,\cdot\,,k_1,\ldots,k_m))\Big)^2 dt\Big]\\
&= \int_0^\infty e^{-(2m-1)t}\sum_{n=d+1}^\infty \Big({n!\over(n-m)!}\Big)^2e^{-2(n-m)t}(n-m)!\lnormb{2}{(n-m)}{f_n(\,\cdot\,,k_1,\ldots,k_m)}^2\, dt\\
&\leq\sum_{n=d+1}^\infty \left({n!\over(n-m)!}\right)^2(n-m)!\,\lnormb{2}{(n-m)}{f_n(\,\cdot\,,k_1,\ldots,k_m)}^2\,,
\end{align*}
where we used that $\int_0^\infty e^{-(2n-1)t}\, dt=(2n-1)^{-1} \leq 1$. Since $F\in{\rm dom}(D^m)$, the latter expression is finite and converges to zero, as $d\to\infty$. This concludes the proof.
\end{proof}

Our next result combines the previous two propositions and is one of the key tools in the proof of our general Berry-Esseen bound in Section \ref{sec:AbstractBerryEsseen}. Similar relations also hold on the Wiener and the Poisson space for which we refer to \cite{NourdinPeccatiReinertSecondOrderPoincare} and \cite{LastPeccatiSchulte}, respectively. Although from a formal point of view the statement looks similar to these results, we emphasize that the proof as well as the meaning and the interpretation of the involved Malliavin operators in our discrete framework are different.

\begin{proposition}\label{OUBeGone}
For $m,k_1,\ldots,k_m\in\N$, $\alpha\geq 1$ and centred $F\in{\rm dom}(D^m)$ one has that $$\E[|D_{k_1,\ldots,k_m}^mL^{-1}F|^\alpha]\leq\E[|D_{k_1,\ldots,k_m}^mF|^\alpha]\,.$$
\end{proposition}

\begin{proof}
According to Proposition \ref{IntMehler}, we have that
\begin{align*}
\E[|D^m_{k_1,\ldots,k_m}L^{-1}F|^\alpha] &= \E\Big[\Big|\int_0^\infty e^{-mt}P_tD_{k_1,\ldots,k_m}^mF\, dt\Big|^\alpha\Big]\,.
\end{align*}
Then, using Proposition \ref{Mehler} together with Jensen's inequality, we deduce that
\begin{align*}
\E\Big[\Big|\int_0^\infty e^{-mt}P_tD_{k_1,\ldots,k_m}^mF\, dt\Big|^\alpha\Big] &= 
\E\Big[\Big| \int_0^\infty e^{-mt} \E[D_{k_1,\ldots,k_m}^m F(X^t) \, \vert \, X] \,  dt \Big|^\alpha\Big]\\
&\leq \E\Big[\int_0^\infty e^{-mt} \E[\vert D_{k_1,\ldots,k_m}^m F(X^t) \vert^\alpha \, \vert \, X] \,  dt \Big]\\
&= \int_0^\infty e^{-mt} \E[\vert D_{k_1,\ldots,k_m}^m F \vert^\alpha] \,  dt\\
&\leq \E[|D_{k_1,\ldots,k_m}^mF|^\alpha]
\end{align*}
and complete the proof. 
\end{proof}

As a first  application of Proposition \ref{OUBeGone} we now deduce the following discrete Poincar\'e-type inequality. This result can already be found in \cite[Chapter 8]{Pri}, where it is proved by means of the Clark formula. We present an alternative proof without resorting to this formula.
\begin{proposition}\label{lem:PoincareUngleichung}
Suppose that $F\in{\rm dom}(D)$. Then
\begin{equation}\label{eq:rhspoincare}
\Var[F] \leq \E[\|DF\|_{\ell^2(\N)}^2]\,.
\end{equation}
\end{proposition}
\begin{proof}
Choosing $f$ in \eqref{eq:IntegrationByParts} as the identity map on $\R$ yields
\begin{align*}
\Var[F] = \E[(F-\E[F])^2] &= \E[\langle D(F-\E[F]), -DL^{-1}(F-\E[F])\rangle_{\ell^2(\N)}]\\ &=\E\Big[\sum_{k=1}^\infty(D_k(F-\E[F]))(-D_kL^{-1}(F-\E[F]))\Big]\\
&\leq\E\Big[\sum_{k=1}^\infty|D_k(F-\E[F])|\,|D_kL^{-1}(F-\E[F])|\Big]\,.
\end{align*}
Exchanging expectation and summation, and using the Cauchy-Schwarz inequality, we see that the latter expression is further bounded by
$$\sum_{k=1}^\infty \big(\E[(D_k(F-\E[F]))^2]\big)^{1/2}\big(\E[(D_kL^{-1}(F-\E[F]))^2]\big)^{1/2}\,.$$ The proof is now concluded by applying Proposition \ref{OUBeGone} with $m=1$ and $\alpha=2$ and using the fact that $D_k(F-\E[F])=D_kF$.
\end{proof}

\begin{remark}\label{rem:pincareL1}
Proposition \ref{lem:PoincareUngleichung} remains valid for $F\in L^1(\Omega)\setminus {\rm dom}(D)$, since in this case the right hand side of \eqref{eq:rhspoincare} is infinite.
\end{remark}

\section{A general Berry-Esseen bound}\label{sec:AbstractBerryEsseen}

The main result of this section is a Berry-Esseen bound for square-integrable Rademacher functionals $F$. By such a result we mean an upper bound for the Kolmogorov distance between $F$ and a standard Gaussian random variable, where we recall that the Kolmogorov distance between two random variables $X$ and $Y$ is defined as $$d_K(X,Y):=\sup_{x\in\R}\big|P(X\leq x)-P(Y\leq x)\big|\,.$$
A first result in this direction has been shown by the authors in \cite{KroReiThae} in the special symmetric case that the sequence $p=(p_k)_{k\in\N}$ is constant and equal to $1/2$. In the present situation, we need the following generalization to arbitrary sequences $p$. Since the proof follows straightforwardly along the lines of that of Theorem 3.1 in \cite{KroReiThae}, we omit the arguments.

\begin{proposition}\label{AbstractBound}
Let $F \in {\rm dom}(D)$ with $\E[F] = 0$ and let $N\sim\mathcal{N}(0,1)$ be a standard Gaussian random variable. Then,
\begin{align*}
d_K(F,N) &\leq \E[|1-\langle DF,-DL^{-1}F\rangle_{\ell^2(\N)}|] + \frac{\sqrt{2\pi}}{8}\E[\langle (pq)^{-1/2} (DF)^2,|DL^{-1}F|\rangle_{\ell^2(\N)}]\\
&\quad +\frac{1}{2}\E[\langle(pq)^{-1/2}(DF)^2,|F\cdot DL^{-1}F|\rangle_{\ell^2(\N)}]\\
&\quad +\sup_{x\in\R}\E[\langle (pq)^{-1/2}(DF)D\1_{\{F>x\}},|DL^{-1}F|\rangle_{\ell^2(\N)}]\,.
\end{align*}
\end{proposition}

One disadvantage of the bound in Proposition \ref{AbstractBound} is that it involves the inverse of the discrete Ornstein-Uhlenbeck operator. In applications this means that the chaotic decomposition of the Rademacher functional $F$ has to be computed explicitly in order to evaluate the expression $-DL^{-1}F$. A further analysis of the bound then requires a multiplication formula for discrete multiple stochastic integrals, which expresses a product of two discrete multiple stochastic integrals as linear combination of discrete multiple stochastic integrals. 
We transfer the bound of Proposition \ref{AbstractBound} into a form, which can be evaluated without using a multiplication formula. Our next result is a combination of Proposition \ref{AbstractBound} and Proposition \ref{OUBeGone}, and provides an upper bound for $d_K(F,N)$ in terms of the first- and second-order discrete gradient only. A result of this structure is what is called a `second-order Poincar\'e inequality' in the literature, see \cite{ChatterjeeSecondOrderPoincare,LastPeccatiSchulte,NourdinPeccatiReinertSecondOrderPoincare}. The discrete Poincar\'e-type inequality in Proposition \ref{lem:PoincareUngleichung} says that a Rademacher functional $F$ is concentrated around $\E[F]$ in terms of the variance if the contribution of the first-order discrete gradient is small. Our discrete second-order Poincar\'e inequality additionally implies that if the contribution of the second-order discrete gradient is also small, then $F$ is close to a standard Gaussian random variable.

\begin{theorem}\label{Folgetheorem}
Let $F\in\text{\rm dom}(D)$ with $\E[F]=0$ and $\E[F^2]=1$, and let $N\sim\mathcal{N}(0,1)$ be a standard Gaussian random variable. Further, fix $r,s,t \in (1,\infty)$ such that $\frac{1}{r}+\frac{1}{s}+\frac{1}{t}=1$. Then,
\begin{align*}
d_K(F,N) &\leq \Big( \frac{15}{4} \sum_{j,k,\ell = 1}^\infty (\E[(D_jF)^2(D_kF)^2])^{1/2}(\E[(D_\ell D_jF)^2(D_\ell D_kF)^2])^{1/2} \Big)^{1/2}\\
&\quad + \Big( \frac{3}{4} \sum_{j,k,\ell = 1}^\infty \frac{1}{p_\ell q_\ell} \E[(D_\ell D_jF)^2(D_\ell D_kF)^2] \Big)^{1/2} + \frac{\sqrt{2\pi}}{8} \sum_{k=1}^\infty \frac{1}{\sqrt{p_k q_k}} \E[ \vert D_kF \vert^3]\\
&\quad + \frac{1}{2} (\E[\vert F \vert^r])^{1/r} \sum_{k=1}^\infty \frac{1}{\sqrt{p_k q_k}} (\E[|D_kF|^{2s}])^{1/s} (\E[\vert D_kF \vert^{t}])^{1/t}\\
&\quad + \Big( \sum_{k=1}^\infty \frac{1}{p_k q_k} \E[(D_kF)^4] \Big)^{1/2} + \Big( 6 \sum_{k,\ell = 1}^\infty \frac{1}{p_k q_k} (\E[(D_kF)^4])^{1/2} (\E[(D_\ell D_kF)^4])^{1/2} \Big)^{1/2}\\
&\quad + \Big( 3 \sum_{k,\ell=1}^\infty \frac{1}{p_k q_k p_\ell q_\ell} \E[(D_\ell D_k F)^4)] \Big)^{1/2}\,.
\end{align*}
\end{theorem}

Let us comment on the second-order Poincar\'e inequality in Theorem \ref{Folgetheorem}. Its form differs from that in the Wiener or Poisson case treated in \cite{LastPeccatiSchulte,NourdinPeccatiReinertSecondOrderPoincare}. The main difference is the fourth term, which involves the parameters $r,s$ and $t$, and hence higher moments of $F$ and $D_kF$. In many applications one can choose $r=2$ and $s=t=4$, but there are situations in which the additional flexibility to choose $r,s$ and $t$ differently turns out to be crucial. We shall meet such an example in the proof of Theorem \ref{thm:ERGraph} on the triangle counting statistic in the Erd\H{o}s-Renyi random graph.

\begin{proof}[Proof of Theorem \ref{Folgetheorem}]
Our proof follows the general scheme to establish a second-order Poin\-car\'e inequality, which is used in the literature \cite{ChatterjeeSecondOrderPoincare,LastPeccatiSchulte,NourdinPeccatiReinertSecondOrderPoincare}. Namely, we build on Proposition \ref{AbstractBound} by further estimating each summand of the bound there. We start with the first summand, to which we apply the Cauchy-Schwarz inequality:
\begin{align*}
\E[|1-\langle DF,-DL^{-1}F\rangle_{\ell^2(\N)}|] & \leq (\E[(1-\langle DF,-DL^{-1}F\rangle_{\ell^2(\N)})^2])^{1/2}.
\end{align*}
Taking $f$ as the identity on $\R$ in \eqref{eq:IntegrationByParts} shows that $\E[\langle DF,-DL^{-1}F\rangle_{\ell^2(\N)}]=\Var[F]=1$. Thus, $\E[(1-\langle DF,-DL^{-1}F\rangle_{\ell^2(\N)})^2] = \Var[\langle DF,-DL^{-1}F\rangle_{\ell^2(\N)}]$ and an application of Proposition  \ref{lem:PoincareUngleichung} (see also Remark \ref{rem:pincareL1}) yields
\begin{align}\label{GradientSummation}
\E[(1-\langle DF,-DL^{-1}F\rangle_{\ell^2(\N)})^2] &\leq \E[\|D(\langle DF,-DL^{-1}F\rangle_{\ell^2(\N)})\|_{\ell^2(\N)}^2]\notag\\
&= \E\Big[\sum_{\ell=1}^\infty \Big(D_\ell\Big(\sum_{k=1}^\infty (D_kF)(-D_kL^{-1}F)\Big)\Big)^2 \Big]\notag \\
&= \E\Big[\sum_{\ell=1}^\infty \Big(\sum_{k=1}^\infty D_\ell ((D_kF)(-D_kL^{-1}F))\Big)^2 \Big]\,,
\end{align}
where the exchange of $D_\ell$ with the summation in the last step can be justified as follows.
Since $\E[\langle DF, -DL^{-1}F \rangle_{\ell^2(\N)}]=1$, $\langle DF, -DL^{-1}F \rangle_{\ell^2(\N)}$ is $P$-a.s.\ finite. Thus, $\langle DF_\ell^{\pm}, -DL^{-1}F_\ell^{\pm} \rangle_{\ell^2(\N)}$ as well as the path-wise representation of $D_\ell(\langle DF, -DL^{-1}F \rangle_{\ell^2(\N)})$ as at \eqref{eq:PathWiseDifferenceOperator} are $P$-a.s.\ finite for $\ell \in \N$. As a consequence, we see that
\begin{align*}
D_\ell(\langle DF, -DL^{-1}F \rangle_{\ell^2(\N)})&= \sqrt{p_\ell q_\ell} (\langle DF_\ell^+, -DL^{-1}F_\ell^+ \rangle_{\ell^2(\N)} - \langle DF_\ell^-, -DL^{-1}F_\ell^- \rangle_{\ell^2(\N)})\\
&= \sqrt{p_\ell q_\ell} \Big(\sum_{k=1}^\infty (D_kF_\ell^+)(-D_kL^{-1}F_\ell^+) - \sum_{k=1}^\infty (D_kF_\ell^-)(-D_kL^{-1}F_\ell^-) \Big)\\
&= \sqrt{p_\ell q_\ell} \sum_{k=1}^\infty ((D_kF_\ell^+)(-D_kL^{-1}F_\ell^+) - (D_kF_\ell^-)(-D_kL^{-1}F_\ell^-))\\
&= \sum_{k=1}^\infty D_\ell((D_kF)(-D_kL^{-1}F))\qquad P\text{-a.s.}
\end{align*}
for $\ell \in \N$.
Now, we further estimate the quantity $D_\ell ((D_kF)(-D_kL^{-1}F))$ in \eqref{GradientSummation} using the product formula \eqref{eq:ProduktFormel}. This yields
\begin{align*}
&|D_\ell\big((D_kF)(-D_kL^{-1}F)\big)|\\
&= |(D_{\ell}D_kF)(-D_kL^{-1}F)+(D_kF)(-D_\ell D_kL^{-1}F)-{X_\ell\over\sqrt{p_\ell q_\ell}}(D_\ell D_kF)(-D_\ell D_kL^{-1}F)|\\
&\leq |(D_{\ell}D_kF)(-D_kL^{-1}F)|+|(D_kF)(-D_\ell D_kL^{-1}F)|+{1\over\sqrt{p_\ell q_\ell}}|(D_\ell D_kF)(-D_\ell D_kL^{-1}F)|\,.
\end{align*}

Using this together with the Cauchy-Schwarz inequality, it follows from \eqref{GradientSummation} that
\begin{align}\label{T_1,T_2,T_3}
\E[(1-\langle DF, -DL^{-1}F \rangle_{\ell^2(\N)})^2] \leq 3(T_1+T_2+T_3)\,,
\end{align}
where $T_1$, $T_2$ and $T_3$ are given by
\begin{align*}
T_1 &:= \E\Big[\sum_{\ell=1}^\infty \Big(\sum_{k=1}^\infty |(D_{\ell}D_kF)(-D_kL^{-1}F)|\Big)^2 \Big]\,,\\
T_2 &:= \E\Big[\sum_{\ell=1}^\infty \Big(\sum_{k=1}^\infty |(D_kF)(-D_\ell D_kL^{-1}F)|\Big)^2 \Big]\,,\\
T_3 &:= \E\Big[\sum_{\ell=1}^\infty {1\over p_\ell q_\ell}\Big(\sum_{k=1}^\infty |(D_\ell D_kF)(-D_\ell D_kL^{-1}F)|\Big)^2 \Big]\,.
\end{align*}
Each of these terms is now further estimated from above. Considering $T_1$, an application of Proposition \ref{IntMehler} and Proposition \ref{Mehler} as well as Jensen's inequality yields for $\ell \in \N$ that
\begin{align*}
\Big(\sum_{k=1}^\infty |D_{\ell}D_kF| \, |D_kL^{-1}F|\Big)^2&= \Big(\sum_{k=1}^\infty |D_{\ell}D_kF| \, \Big| \int_0^\infty e^{-t}P_tD_kF \,  dt \Big| \Big)^2\\
&= \Big(\sum_{k=1}^\infty |D_{\ell}D_kF| \, \Big| \int_0^\infty e^{-t}\E[D_kF(X^t) \, \vert \, X] \, d t \Big| \Big)^2\\
&\leq \Big(\sum_{k=1}^\infty |D_{\ell}D_kF| \int_0^\infty e^{-t}\E[|D_kF(X^t)| \, \vert \, X] \,  dt \Big)^2\,.
\end{align*}
By virtue of the monotone convergence theorem, we get for $\ell \in \N$ that
\begin{align*}
&\Big(\sum_{k=1}^\infty |D_{\ell}D_kF| \int_0^\infty e^{-t}\E[|D_kF(X^t)| \, \vert \, X] \, dt \Big)^2\\
&= \Big(\int_0^\infty e^{-t} \sum_{k=1}^\infty |D_{\ell}D_kF| \E[|D_kF(X^t)| \, \vert \, X] \,  dt \Big)^2\\
&= \Big(\int_0^\infty e^{-t} \E\Big[ \sum_{k=1}^\infty |D_{\ell}D_kF| \, |D_kF(X^t)| \, \Big| \, X\Big] \,  dt \Big)^2\,.
\end{align*}
Using Jensen's inequality again as well as the Cauchy-Schwarz inequality, we now conclude for $\ell \in \N$ that
\begin{align*}
&\Big(\int_0^\infty e^{-t} \E\Big[ \sum_{k=1}^\infty |D_{\ell}D_kF| \, |D_kF(X^t)| \, \Big| \, X\Big] \,  dt \Big)^2\\
&\leq \int_0^\infty e^{-t} \E\Big[ \Big( \sum_{k=1}^\infty |D_{\ell}D_kF| \, |D_kF(X^t)| \Big)^2 \, \Big| \, X\Big] \,  dt\\
&= \int_0^\infty e^{-t} \E\Big[ \sum_{j,k=1}^\infty |D_{\ell}D_jF| \, |D_jF(X^t)| \, |D_{\ell}D_kF| \, |D_kF(X^t)| \, \Big| \, X\Big] \,  dt\\
&= \sum_{j,k=1}^\infty |D_{\ell}D_jF| \, |D_{\ell}D_kF| \int_0^\infty e^{-t} \E[ |D_jF(X^t)| \, |D_kF(X^t)| \, | \, X] \,  dt\\
&\leq \sum_{j,k=1}^\infty |D_{\ell}D_jF| \, |D_{\ell}D_kF| \int_0^\infty e^{-t} (\E[ (D_jF(X^t))^2 (D_kF(X^t))^2 \, | \, X])^{1/2} \,  dt\\
&\leq \sum_{j,k=1}^\infty |D_{\ell}D_jF| \, |D_{\ell}D_kF| \Big( \int_0^\infty e^{-t} \E[ (D_jF(X^t))^2 (D_kF(X^t))^2 \, | \, X] \,  dt \Big)^{1/2}\,.
\end{align*}
Thus, another application of the Cauchy-Schwarz inequality leads to the bound
\begin{align}\label{T_1}
T_1 &\leq \E \Big[ \sum_{j,k,\ell=1}^\infty |D_{\ell}D_jF| \, |D_{\ell}D_kF| \Big( \int_0^\infty e^{-t} \E[ (D_jF(X^t))^2 (D_kF(X^t))^2 \, | \, X] \,  dt \Big)^{1/2} \Big] \notag\\
&\leq \sum_{j,k,\ell=1}^\infty (\E[(D_{\ell}D_jF)^2(D_{\ell}D_kF)^2])^{1/2} \Big( \E \Big[ \int_0^\infty e^{-t} \E[ (D_jF(X^t))^2 (D_kF(X^t))^2 \, | \, X] \,  dt \Big] \Big)^{1/2} \notag\\
&= \sum_{j,k,\ell=1}^\infty (\E[(D_{\ell}D_jF)^2(D_{\ell}D_kF)^2])^{1/2} \Big( \int_0^\infty e^{-t} \E[(D_jF)^2(D_kF)^2] \,  dt \Big)^{1/2} \notag\\
&= \sum_{j,k,\ell=1}^\infty (\E[(D_{\ell}D_jF)^2(D_{\ell}D_kF)^2])^{1/2}(\E[(D_jF)^2(D_kF)^2])^{1/2}\,.
\end{align}
Using similar arguments and Proposition \ref{IntMehler} for $m=2$, one shows that
\begin{align}\label{T_2}
T_2 \leq \frac{1}{4} \sum_{j,k,\ell=1}^\infty (\E[(D_jF)^2(D_kF)^2])^{1/2}(\E[(D_{\ell}D_jF)^2(D_{\ell}D_kF)^2])^{1/2}
\end{align}
and
\begin{align}\label{T_3}
T_3 \leq \frac{1}{4} \sum_{j,k,\ell=1}^\infty \frac{1}{p_\ell q_\ell} \E[(D_{\ell}D_jF)^2(D_{\ell}D_kF)^2]\,.
\end{align}
Thus, combining \eqref{T_1}, \eqref{T_2} and \eqref{T_3} with \eqref{T_1,T_2,T_3}, we get
\begin{align}\label{FirstSummand}
&\E[|1-\langle DF,-DL^{-1}F\rangle_{\ell^2(\N)}|] \notag\\
&\leq \Big( \frac{15}{4} \sum_{j,k,\ell = 1}^\infty (\E[(D_jF)^2(D_kF)^2])^{1/2}(\E[(D_\ell D_jF)^2(D_\ell D_kF)^2])^{1/2} \Big)^{1/2} \notag\\
&\quad + \Big( \frac{3}{4} \sum_{j,k,\ell = 1}^\infty \frac{1}{p_\ell q_\ell} \E[(D_\ell D_jF)^2(D_\ell D_kF)^2] \Big)^{1/2}
\end{align}
as an estimate for the first summand of the bound in Proposition \ref{AbstractBound}. 

For the second summand we obtain
\begin{align}
\E[\langle (pq)^{-1/2} (DF)^2,|DL^{-1}F|\rangle_{\ell^2(\N)}] &= \sum_{k=1}^\infty (p_k q_k)^{-1/2} \E[(D_kF)^2 \, |D_kL^{-1}F|]\notag\\
&\leq \sum_{k=1}^\infty (p_k q_k)^{-1/2} (\E[|D_kF|^3])^{2/3} (\E[|D_kL^{-1}F|^3])^{1/3}\notag\\
&\leq \sum_{k=1}^\infty (p_k q_k)^{-1/2} \E[|D_kF|^3]
\end{align}
by means of H\"older's inequality with H\"older conjugates $3$ and $3/2$, and Proposition \ref{OUBeGone}. Applying a generalization of H\"older's inequality with H\"older conjugates $r,s,t \in (1,\infty)$ with $\frac{1}{r}+\frac{1}{s}+\frac{1}{t}=1$ as well as Proposition \ref{OUBeGone} to the third summand of the bound in Proposition \ref{AbstractBound} immediately yields
\begin{align}
&\E[\langle(pq)^{-1/2}(DF)^2,|F\cdot DL^{-1}F|\rangle_{\ell^2(\N)}]\notag\\
&= \sum_{k=1}^\infty (p_k q_k)^{-1/2} \E[|F| \, (D_kF)^2 \, |D_kL^{-1}F|]\notag\\
&\leq (\E[|F|^r])^{1/r} \sum_{k=1}^\infty (p_k q_k)^{-1/2} (\E[|D_kF|^{2s}])^{1/s} (\E[|D_kL^{-1}F|^t])^{1/t}\notag\\
&\leq (\E[|F|^r])^{1/r} \sum_{k=1}^\infty (p_k q_k)^{-1/2} (\E[|D_kF|^{2s}])^{1/s} (\E[|D_kF|^t])^{1/t}\,.
\end{align}

We now apply the integration-by-parts-formula \eqref{Indicator adjointness} in order to bound the last term in Proposition \ref{AbstractBound}. To this end we note that $D_k\1_{\{F>x\}}D_kF|D_kL^{-1}F|\geq 0$ for every $k\in\N$ and we need to verify the summability condition in \eqref{Indicator adjointness u} in Proposition \ref{Indicator adjointness lemma}. The latter will be verified subsequent to the following calculation. Using the integration-by-parts-formula we obtain that
\begin{align}\label{IndicatorBeGone}
\E[\langle (pq)^{-1/2}(DF)D\1_{\{F>x\}},|DL^{-1}F| \rangle_{\ell^2(\N)}]&= \E[\langle D\1_{\{F>x\}}, (pq)^{-1/2}(DF) |DL^{-1}F| \rangle_{\ell^2(\N)}]\notag\\
&= \E[\1_{\{F>x\}} \delta((pq)^{-1/2}(DF) |DL^{-1}F|)]\notag\\
&\leq \E[|\delta((pq)^{-1/2}(DF) |DL^{-1}F|)|]\notag\\
&\leq (\E[(\delta((pq)^{-1/2}(DF) |DL^{-1}F|))^2])^{1/2}\,.
\end{align}
From the isometric formula \eqref{eq:SkorohodIsometry} for the divergence operator it follows that
\begin{align}\label{T_4,T_5}
&\E[(\delta((pq)^{-1/2}(DF) |DL^{-1}F|))^2]\notag\\
&= \E[\| (pq)^{-1/2}(DF) (DL^{-1}F) \|_{\ell^2(\N)}^2]\notag\\
&\quad + \E\Big[ \sum_{k,\ell=1}^\infty (p_\ell q_\ell)^{-1/2} D_k((D_\ell F) |D_\ell L^{-1}F|) \cdot (p_k q_k)^{-1/2} D_\ell((D_k F) |D_k L^{-1}F|) \Big]\notag\\
&\leq \E[\| (pq)^{-1/2}(DF) (DL^{-1}F) \|_{\ell^2(\N)}^2]+ \E\Big[ \sum_{k,\ell=1}^\infty (p_k q_k)^{-1} (D_\ell((D_k F) (D_k L^{-1}F)))^2 \Big]\notag\\
&=: T_4 + T_5\,.
\end{align}
The term $T_4$ can easily be estimated by means of the Cauchy-Schwarz inequality and Proposition \ref{OUBeGone}, which yields that
\begin{align}\label{T_4}
T_4 &= \sum_{k=1}^\infty (p_k q_k)^{-1} \E[(D_kF)^2 (D_kL^{-1}F)^2]
\leq \sum_{k=1}^\infty (p_k q_k)^{-1} (\E[(D_kF)^4])^{1/2} (\E[(D_kL^{-1}F)^4])^{1/2}\notag\\
&\leq \sum_{k=1}^\infty (p_k q_k)^{-1} \E[(D_kF)^4]\,.
\end{align}
To handle $T_5$, we first compute $\E[(D_\ell((D_k F) (D_k L^{-1}F)))^2]$ by using the product formula \eqref{eq:ProduktFormel}, the Cauchy-Schwarz inequality as well as Proposition \ref{OUBeGone}. This leads to
\begin{align}\label{T_5}
&\E[(D_\ell((D_k F) (D_k L^{-1}F)))^2]\notag\\
&= \E[((D_\ell D_kF)(D_k L^{-1}F) + (D_kF)(D_\ell D_k L^{-1}F) - (p_\ell q_\ell)^{-1/2}X_\ell (D_\ell D_kF)(D_\ell D_k L^{-1}F))^2]\notag\\
&\leq 3\E[(D_\ell D_kF)^2(D_k L^{-1}F)^2] + 3\E[(D_kF)^2(D_\ell D_k L^{-1}F)^2]\notag\\
&\quad +3 \, (p_\ell q_\ell)^{-1}\E[(D_\ell D_kF)^2(D_\ell D_k L^{-1}F)^2]\notag\\
&\leq 3 \, (\E[(D_\ell D_kF)^4])^{1/2}(\E[(D_k L^{-1}F)^4])^{1/2} + 3 \, (\E[(D_kF)^4])^{1/2}(\E[(D_\ell D_k L^{-1}F)^4])^{1/2}\notag\\
&\quad +3 \, (p_\ell q_\ell)^{-1}(\E[(D_\ell D_kF)^4])^{1/2}(\E[(D_\ell D_k L^{-1}F)^4])^{1/2}\notag\\
&\leq 6 \, (\E[(D_kF)^4])^{1/2}(\E[(D_\ell D_k F)^4])^{1/2} + 3 \, (p_\ell q_\ell)^{-1}\E[(D_\ell D_kF)^4]\,.
\end{align}
We now justify the validity of the summability condition \eqref{Indicator adjointness u}. Assume that 
\begin{equation}
\label{finiteDoubleGrad}
\E\bigg[\sum_{k,\ell=1}^{\infty}(D_{k}u_\ell)^2\bigg]<\infty\,,
\end{equation}
where $u_\ell:=(p_\ell q_\ell)^{-1/2}\,D_\ell F|D_\ell L^{-1}F|=\sum_{n=1}^{\infty}J_{n}(g_{n+1}(\,\cdot\,,\ell))$. Then we obtain that
\begin{align*}
\E\bigg[\sum_{k,\ell=1}^{\infty}(D_{k}u_\ell)^2\bigg]&=\sum_{k,\ell=1}^{\infty}\E\big[(D_ku_\ell)^2\big]\\
&=\sum_{k,\ell=1}^{\infty}\sum_{n=1}^{\infty}n\, n!\lnorm{2}{n-1}{g_{n+1}(\,\cdot\,,k,\ell)}^2\\
&=\sum_{n=1}^{\infty}n\, n!\lnorm{2}{n+1}{g_{n+1}}^2,
\end{align*}
which implies that 
\[\sum_{n=1}^\infty(n+1)\, n!\lnorm{2}{n+1}{g_{n+1}}^2\leq 2\,\E\Big[\sum_{k,\ell=1}^{\infty}(D_{k}u_\ell)^2\Big]<\infty\,.\]
Thus, the summability condition \eqref{Indicator adjointness u} is verified, once condition \eqref{finiteDoubleGrad} is satisfied. Since $T_5=\E\big[\sum_{k,\ell=1}^\infty(D_{k}u_\ell)^2\big]$, condition \eqref{finiteDoubleGrad} is verified, once our error bound is finite. Otherwise, the error bound holds trivially.
Combining \eqref{IndicatorBeGone}, \eqref{T_4,T_5}, \eqref{T_4} and \eqref{T_5} yields
\begin{align}
&\sup_{x \in \R} \E[\langle (pq)^{-1/2}(DF)D\1_{\{F>x\}},|DL^{-1}F| \rangle_{\ell^2(\N)}]\notag\\
&\leq \Big( \sum_{k=1}^\infty (p_k q_k)^{-1} \E[(D_kF)^4] \Big)^{1/2} +	\Big( 6 \sum_{k,\ell=1}^\infty (p_k q_k)^{-1} (\E[(D_kF)^4])^{1/2} (\E[(D_\ell D_kF)^4])^{1/2} \Big] \Big)^{1/2}\notag\\
&\quad + \Big( 3 \sum_{k,\ell=1}^\infty \frac{1}{p_k q_k p_\ell q_\ell} \E[(D_\ell D_k F)^4] \Big)^{1/2}\,.
\end{align}
This concludes the proof.
\end{proof}

\section{Application to the Erd\H{o}s-R\'enyi random graph\\ and proof of Theorems \ref{thm:ERGraph}, \ref{thm:Subgraphs} and \ref{thm:degrees}}\label{sec:RandomGraphs}

In this section we apply Theorem \ref{Folgetheorem} to counting statistics associated with the \textit{Erd\H{o}s-R\'enyi} random graph and establish thereby Theorem \ref{thm:ERGraph}, Theorem \ref{thm:Subgraphs} and Theorem \ref{thm:degrees}. First, we formally introduce the model and fix some notation. For $n\in\N$ and a real number $p\in(0,1)$, let $\mathcal{G}$ be the set of all simple and undirected graphs with vertex set $[n]:=\{1,\dots,n\}$. We consider the probability space $(\mathcal{G},\mathcal{P}(\mathcal{G}),\mathds{P})$, where $\mathcal{P}(\mathcal{G})$ is
the power set of $\mathcal{G}$ and $\mathds{P}$ is the probability measure given by
\[\mathds{P}(G)=p^{e(G)}(1-p)^{\binom{n}{2}-e(G)}\,,\] 
where for $G\in\mathcal{G}$, $e(G)$ denotes the number of edges of $G$. The identity map on $\mathcal{G}$ is called the Erd\H{o}s-R\'enyi random graph and is usually abbreviated by $G(n,p)$. We refer to the book \cite{JanLucRuc} for a detailed account of the theory of random graphs.

We are interested in the number $T$ of triangles in $G(n,p)$, that is the number of sub-graphs in $G(n,p)$, which are isomorphic to the complete graph on $3$ vertices. To analyse the asymptotic behaviour of this random variable, we typically allow $p$ to depend on $n$. Following the literature and to simplify the notation we will often suppress the dependence on $n$ of several (random) variables. In order to apply Theorem \ref{Folgetheorem} to the normalized triangle counting statistic $F:=(T-\E[T])/\sqrt{\Var[T]}$, we first have to embed the model into the framework of Section \ref{sec:Preliminaries} and Section \ref{sec:AbstractBerryEsseen}. If one labels the $\binom{n}{2}$ edges of the complete graph on $n$ vertices in a fixed but arbitrary way, $G(n,p)$ can be regarded as an outcome of $\binom{n}{2}$ independent Bernoulli trials, with success probability equal to $p$. Here, success in the $k$'th Bernoulli trial means that the $k$'th edge is visible in $G(n,p)$. Hence, $G(n,p)$ can be identified with the vector $\big(X_1,\dots,X_{\binom{n}{2}}\big)$ of independent Rademacher random variables with parameter $p_k\equiv p$, where $X_k=+1$ indicates that the edge with number $k$ is visible in $G(n,p)$. From now on, we fix an arbitrary enumeration of the edges in the complete graph on the vertex set $[n]$, write $I:=\{1,\dots,\binom{n}{2}\}$ for the set of labels and denote by $e_k$, $k\in I$, the $k$'th edge of the graph.

Recall from Lemma 3.5 in \cite{JanLucRuc} that 
\begin{equation}
 \label{eq:triangleVariance}
\Var[T]\asymp\begin{cases}\theta^5n^{4-5\alpha}&\text{if }0\leq\alpha\leq\frac 12\\\theta^3n^{3(1-\alpha)}&\text{if }\frac 12<\alpha<1\,,\end{cases}
\end{equation}
where we recall that the success probability is given by $p=\theta n^{-\alpha}$ with $\alpha\in[0,1)$ and $\theta\in(0,n^{\alpha})$ such that $\theta\asymp 1$.

\begin{proof}[Proof of Theorem \ref{thm:ERGraph}]
First, we notice that the assumptions of Theorem \ref{Folgetheorem} are satisfied since $F$ is normalized and only depends on finitely many Rademacher variables.

To evaluate the bound in Theorem \ref{Folgetheorem}, we have to control the random variables $D_kF$ and $D_kD_jF$ for $k,j\in\{1,\ldots,\binom{n}{2}\}$. We have
\begin{align*}
 D_kF&=\sqrt{pq}(F_k^+-F_k^-)=\frac{\sqrt{pq}}{\sqrt{\Var[T]}}(T_k^+-T_k^-)
\end{align*}
and hence
\begin{align*}
 \frac{\sqrt{\Var[T]}}{\sqrt{pq}}D_kF=T_k^+-T_k^-\,.
\end{align*}
Now, we notice that $T_k^+$ equals the number of triangles in the random graph $G(n,p)$ with $e_k$ visible, while $T_k^-$ equals the number of triangles in the random graph $G(n,p)$ when $e_k$ is not visible. Thus,
$T_k^+-T_k^-$ is the number of triangles that have edge $e_k$ in common, which shows that the random variable $T_k^+-T_k^-$
has a binomial distribution $\text{Bin}(n-2,p^2)$ with parameters $n-2$ and $p^2$. This is a consequence of the fact that there are $n-2$ possible triangles being attached to the $k$'th edge and each of these triangles is a sub-graph of $G(n,p)$ with probability $p^2$, independently of all other triangles. Hence,
\begin{equation}
 \label{eq:DkT}
\frac{\sqrt{\Var[T]}}{\sqrt{pq}}D_kF\sim\text{Bin}(n-2,p^2)\,.
\end{equation}
Next, we consider the second-order discrete gradient and obtain that
\begin{align*}
 D_kD_jF&=\frac{\sqrt{pq}}{\sqrt{\Var[T]}}D_k(T_j^+-T_j^-)\\
&=\frac{pq}{\sqrt{\Var[T]}}\big((T_j^+)_k^+-(T_j^+)_k^--\big((T_j^-)_k^+-(T_j^-)_k^-\big)\big)\,,
\end{align*}
whence
\begin{align*}
 \frac{\sqrt{\Var[T]}}{pq}D_kD_jF=(T_j^+)_k^+-(T_j^+)_k^--\big((T_j^-)_k^+-(T_j^-)_k^-\big)\,.
\end{align*}
The random variable $(T_j^+)_k^+-(T_j^+)_k^-$ counts the number of triangles in $G(n,p)$ adjacent to the $k$'th edge $e_k$, conditioned on the event that the $j$'th edge $e_j$ is visible in $G(n,p)$. Similarly, $(T_j^-)_k^+-(T_j^-)_k^-$ counts the number of triangles adjacent to $e_k$ when $e_j$ does not belong to $G(n,p)$. Therefore, $(T_j^+)_k^+-(T_j^+)_k^--\big((T_j^-)_k^+-(T_j^-)_k^-\big)$ is the number of triangles with common edges $e_k$ and $e_j$. The number of vertices shared by both edges $e_k$ and $e_j$ is $|e_k\cap e_j|$. Then, if $|e_k\cap e_j|\in\{0,2\}$, $(T_j^+)_k^+-(T_j^+)_k^--\big((T_j^-)_k^+-(T_j^-)_k^-\big)=0$ and if $|e_k\cap e_j|=1$, we have $e_k=\{r,s\}$ and $e_j=\{r,t\}$ for some $r,s,t\in[n]$, $s\neq t$. In this case, $(T_j^+)_k^+-(T_j^+)_k^--\big((T_j^-)_k^+-(T_j^-)_k^-\big)$  is either $1$ or $0$, depending on whether the edge $\{s,t\}$ belongs to $G(n,p)$ or not. Thus,
\begin{align}
\label{eq:DkDjT}
 \frac{\sqrt{\Var[T]}}{pq}D_kD_jF\begin{cases}\sim\text{Ber}(p)&\text{if }|e_k\cap e_j|=1\\=0&\text{if }|e_k\cap e_j|\in\{0,2\}\,,\end{cases}
\end{align}
where ${\rm Ber}(p)={\rm Bin}(1,p)$ indicates a Bernoulli distribution with parameter $p$.
Note that the random variables $D_\ell D_kF$ and $D_\ell D_jF$ are independent whenever $k\neq j$. Indeed, fix $\ell$ and let $k\neq j$, and suppose that $|e_k\cap e_\ell |\in\{0,2\}$ or $|e_j\cap e_\ell |\in\{0,2\}$. Then $D_\ell D_kF$ and $D_\ell D_jF$ are independent, since at least one of them is equal to zero. Now, consider the case that $|e_k\cap e_\ell |=1$ and $|e_j\cap e_\ell |=1$. In this situation, the three edges $e_k,\,e_j,\,e_\ell$ can have the following form. Either
\begin{equation}
\label{eq:case1}
e_k=\{s,t\}\,,\quad e_j=\{u,v\}\,,\quad e_\ell =\{t,u\}\,,\qquad \,s\neq u,\,v\neq t\,,
\end{equation}
or
\begin{equation}
\label {eq:case2}
e_k=\{s,t\}\,,\quad e_j=\{u,t\}\,,\quad e_\ell =\{v,t\}\,,\qquad s\neq v,\,u\neq v\,.
\end{equation}
In the situation of \eqref{eq:case1}, we have $\{s,u\}=e_a$ and $\{t,v\}=e_b$ for some $a,b\in I$, $a\neq b$ and thus
\[\frac{\sqrt{\Var[T]}}{pq}D_\ell D_kF=\1_{\{X_a=1\}}\quad\text{and}\quad\frac{\sqrt{\Var[T]}}{pq}D_\ell D_jF=\1_{\{X_b=1\}}\,,\]
which implies the independence of $D_\ell D_kF$ and $D_\ell D_jF$ in this case. In the situation of \eqref{eq:case2} we obtain $\{s,v\}=e_a$ and $\{u,v\}=e_b$ for some $a,b\in I$, $a\neq b$, and  hence
\[\frac{\sqrt{\Var[T]}}{pq}D_\ell D_kF=\1_{\{X_a=1\}}\quad\text{and}\quad\frac{\sqrt{\Var[T]}}{pq}D_\ell D_jF=\1_{\{X_b=1\}}\,,\]
which implies the independence of $D_\ell D_kF$ and $D_\ell D_jF$ in the second case.

In view of \eqref{eq:DkT} and the bound in Theorem \ref{Folgetheorem} we need an expression for the fractional moments of a binomial random variable $Z\sim\text{Bin}(n,p)$ with $n\in\N$ and $p\in(0,1)$. It is well known that
\begin{equation*}
\E[Z^\beta]\asymp\begin{cases}(np)^\beta&\text{if }np\to\infty\\np&\text{if }np\to 0\,,\end{cases}\qquad\qquad \beta\in[1,\infty)\,.
\end{equation*}
As a consequence, we deduce that for $n\in\{3,4,\ldots\}$, $\alpha\in[0,1)$ and $\theta\in(0,n^{\alpha})$ with $\theta\asymp 1$, the binomial random variable $Z\sim\text{Bin}(n-2,\theta^2n^{-2\alpha})$ satisfies
\begin{equation}
 \label{eq:momentBinom}
\E[Z^\beta]\asymp\begin{cases}\theta^\beta n^{\beta(1-2\alpha)}&\text{if }0\leq\alpha\leq\frac 12\\\theta\, n^{1-2\alpha}&\text{if }\frac 12<\alpha< 1\,,\end{cases}\qquad\qquad \beta\in[1,\infty)\,.
\end{equation}

With \eqref{eq:DkT}, \eqref{eq:DkDjT} and \eqref{eq:momentBinom} at hand we are now prepared for the evaluation of the bound in Theorem \ref{Folgetheorem}. The following terms have to be considered:
\begin{alignat*}{2}
 A_1&:= \Big(\sum_{j,k,\ell\in I}(\E[(D_jF)^2(D_kF)^2])^{1/2}(\E[(D_\ell D_jF)^2(D_\ell D_kF)^2])^{1/2} \Big)^{1\over 2}\,,\\
A_2&:=\Big(\sum_{j,k,\ell\in I}\frac{1}{pq} \E[(D_\ell D_jF)^2(D_\ell D_kF)^2] \Big)^{1\over 2}\,, && \hspace{-2cm} A_3:=\sum_{k\in I}\frac{1}{\sqrt{pq}} \E[ \vert D_kF \vert^3]\,,\\
A_4&:=(\E[\vert F \vert^r])^{1\over r} \sum_{k\in I}\frac{1}{\sqrt{pq}} (\E[|D_kF|^{2s}])^{1\over s} (\E[\vert D_kF \vert^{t}])^{1\over t}\,, && \hspace{-2cm} A_5:=\Big( \sum_{k\in I} \frac{1}{pq} \E[(D_kF)^4] \Big)^{1\over 2}\,,\\
A_6&:=\Big( \sum_{k,\ell\in I}\frac{1}{pq} (\E[(D_kF)^4])^{1/2} (\E[(D_\ell D_kF)^4])^{1/2} \Big)^{1\over 2}\,, && \hspace{-2cm}
A_7:=\frac{1}{p q}\Big(\sum_{k,\ell\in I}  \E[(D_\ell D_k F)^4)] \Big)^{1\over 2}\,,
\end{alignat*}
where in $A_4$, $r,s,t \in (1,\infty)$ are such that $\frac{1}{r}+\frac{1}{s}+\frac{1}{t}=1$. 
Let us begin with the term $A_1$. Using the independence of $D_\ell D_kF$ and $D_\ell D_jF$ for $k\neq j$ as well as the Cauchy-Schwarz inequality we obtain
\begin{align}
&\quad \sum_{j,k,\ell\in I}(\E[(D_jF)^2(D_kF)^2])^{1/2}(\E[(D_\ell D_jF)^2(D_\ell D_kF)^2])^{1/2}\notag\\
&=\sum_{j,\ell\in I}(\E[(D_jF)^4])^{1/2}(\E[(D_\ell D_jF)^4])^{1/2}\notag\\
&\qquad+\sum_{\substack{j,k,\ell\in I\\ k\neq j}}(\E[(D_jF)^2(D_kF)^2])^{1/2}(\E[(D_\ell D_jF)^2])^{1/2}(\E[(D_\ell D_kF)^2])^{1/2}\notag\\
&\leq\sum_{j,\ell\in I}(\E[(D_jF)^4])^{1/2}(\E[(D_\ell D_jF)^4])^{1/2}\notag\\
&\qquad+\sum_{\substack{j,k,\ell\in I\\ k\neq j}}(\E[(D_jF)^4])^{1/4}(\E[(D_kF)^4])^{1/4}(\E[(D_\ell D_jF)^2])^{1/2}(\E[(D_\ell D_kF)^2])^{1/2}.\label{eq:A1a}
\end{align}
We consider the two summands of the last estimate separately. Denote by $\mu_4$ the fourth moment of a $\text{Bin}(n-2,p^2)$-distributed random variable. Using \eqref{eq:DkT} and \eqref{eq:DkDjT}, we see that
\begin{align}
 &\quad\sum_{j,\ell\in I}(\E[(D_jF)^4])^{1/2}(\E[(D_\ell D_jF)^4])^{1/2}\notag\\
&=\frac{(pq)^3}{(\Var[T])^2}\sum_{j,\ell\in I}\Big(\E\Big[\Big(\frac{\sqrt{\Var[T]}}{\sqrt{pq}}D_jF\Big)^4\Big]\Big)^{1/2}\Big(\E\Big[\Big(\frac{\sqrt{\Var[T]}}{pq}D_\ell D_jF\Big)^4\Big]\Big)^{1/2}\notag\\
&=\frac{(pq)^3}{(\Var[T])^2}\sum_{j,\ell\in I}\mu_4^{1/2}p^{1/2}\1_{\{|e_j\cap e_\ell |=1\}}\notag\\
&=\frac{(pq)^3}{(\Var[T])^2}\mu_4^{1/2}p^{1/2}\binom{n}{2}2(n-2)\notag\\
&\asymp\frac{(pq)^3}{(\Var[T])^2}\mu_4^{1/2}p^{1/2}n^3\,.\label{eq:A1b}
\end{align}
For the second summand on the right hand side of \eqref{eq:A1a} we obtain 
\begin{align}
&\quad \sum_{\substack{j,k,\ell\in I\\ k\neq j}}(\E[(D_jF)^4])^{1/4}(\E[(D_kF)^4])^{1/4}(\E[(D_\ell D_jF)^2])^{1/2}(\E[(D_\ell D_kF)^2])^{1/2}\notag\\
&=\frac{(pq)^3}{(\Var[T])^2}\sum_{\substack{j,k,\ell\in I\\j\neq k}}\mu_4^{1/2}p\1_{\{|e_j\cap e_\ell |=1\}}\1_{\{|e_k\cap e_\ell |=1\}}\notag\\
&=\frac{(pq)^3}{(\Var[T])^2}\mu_4^{1/2}p\binom{n}{2}2(n-2)(2(n-2)-1)\notag\\
&\asymp\frac{(pq)^3}{(\Var[T])^2}\mu_4^{1/2}p\,n^4\notag\\
&=\frac{(pq)^3}{(\Var[T])^2}\mu_4^{1/2}p^{1/2}n^3\,p^{1/2}n\,.\label{eq:A1c}
\end{align}
Comparing \eqref{eq:A1b} with \eqref{eq:A1c} one can see that the second summand in \eqref{eq:A1a} determines the asymptotic behaviour of $A_1$, since $p^{1/2}n=\theta^{1/2}n^{1-\alpha/2}\to\infty$, as $n\to\infty$. By use of \eqref{eq:triangleVariance} and \eqref{eq:momentBinom} we obtain
\begin{equation}
\label{eq:A1d}
\frac{(pq)^3}{(\Var[T])^2}\mu_4^{1/2}p\,n^4\asymp\begin{cases}\theta^{-4}\, n^{-2+2\alpha}&\text{if }0\leq\alpha\leq\frac 12\\\theta^{-{3\over 2}}\, n^{-{3\over 2}+\alpha}&\text{if }\frac 12<\alpha< 1\,.\end{cases}
\end{equation}
Combining \eqref{eq:A1a}, \eqref{eq:A1b}, \eqref{eq:A1c} and \eqref{eq:A1d} yields that
\begin{align}
 A_1&= \Big(\sum_{j,k,\ell\in I}(\E[(D_jF)^2(D_kF)^2])^{1/2}(\E[(D_\ell D_jF)^2(D_\ell D_kF)^2])^{1/2} \Big)^{1/2}\notag\\
&=\begin{cases}\mathcal{O}\big(n^{-1\,+\,\alpha}\big)&\text{if }0\leq \alpha\leq\frac 12\\\mathcal{O}\big(n^{-{3\over 4}+{\alpha\over 2}}\big)&\text{if }\frac 12<\alpha< 1\,.\end{cases}\label{eq:A1}
\end{align}
With the same arguments as above and by using the additional information on the asymptotics of the third moment of a Bin($n-2,p^2$) random variable from \eqref{eq:momentBinom}, we obtain the following bounds for $A_2$, $A_3$, $A_5$, $A_6$ and $A_7$:
\begin{alignat}{4}
 A_2&=\begin{cases}\mathcal{O}\big(n^{-2+{5\alpha\over 2}}\big) &\text{if }0\leq\alpha\leq\frac 12\\\mathcal{O}\big(n^{-1+{\alpha\over 2}}\big)&\text{if }\frac 12<\alpha< 1\,,\end{cases}
 && \qquad A_3=\begin{cases}\mathcal{O}\big(n^{-1+{\alpha\over 2}}\big)&\text{if }0\leq\alpha\leq\frac 12\\\mathcal{O}\big(n^{-{3\over 2}+{3\alpha\over 2}}\big)&\text{if }\frac 12<\alpha< 1\,,\end{cases}\label{eq:A3}\\
 A_5&=\begin{cases}\mathcal{O}\big(n^{-1+{\alpha\over 2}}\big)&\!\text{if }0\leq\alpha\leq\frac 12\\\mathcal{O}\big(n^{-{3\over 2}+{3\alpha\over 2}}\big)&\!\text{if }\frac 12<\alpha< 1\,,\end{cases}
 &&\qquad  A_6=\begin{cases}\mathcal{O}\big(n^{-{3\over 2}+{7\alpha\over 4}}\big)&\text{if }0\leq\alpha\leq\frac 12\\\mathcal{O}\big(n^{-{5\over 4}+{5\alpha\over 4}}\big)&\text{if }\frac 12<\alpha< 1\,,\end{cases}\label{eq:A6}\\
A_7&=\begin{cases}\mathcal{O}\big(n^{-{5\over 2}+{\frac{7\alpha}{2}}}\big)&\text{if }0\leq\alpha\leq\frac 12\\\mathcal{O}\big(n^{-{3\over 2}+\frac{3\alpha}{2}}\big)&\text{if }\frac 12<\alpha< 1\,.\end{cases}\label{eq:A7}
\end{alignat}

To describe the asymptotic behaviour of
\[A_4=(\E[\vert F \vert^r])^{1/r} \sum_{k\in I}\frac{1}{\sqrt{pq}} (\E[|D_kF|^{2s}])^{1/s} (\E[\vert D_kF \vert^{t}])^{1/t}\]
with $r,s,t \in (1,\infty)$ and $\frac{1}{r}+\frac{1}{s}+\frac{1}{t}=1$, we use the following moment asymptotics, which is taken from the proof of \cite[Theorem 2]{Ruc}. As $n\to\infty$, it holds that

\begin{equation}
\label{eq:rmoments}
\E[F^k]\asymp\begin{cases}
0&\text{if }k\in\N\text{ is odd}\\
\frac{k!}{(k/2)!2^{k/2}}&\text{if }k\in\N\text{ is even}\,.
\end{cases}
\end{equation}

We will choose $r$ in such a way that $A_4$ converges to zero at least as fast as all the other terms $A_1,\ldots,A_7$ that have already been computed. So, fix an even integer $r>2$ and choose $s,t\in(1,\infty)$ such that $\frac{1}{r}+\frac{1}{s}+\frac{1}{t}=1$. For $\beta\in[1,\infty)$ let $\mu_\beta$ be the moment of order $\beta$ of a Bin$(n-2,p^2)$ random variable. Using \eqref{eq:DkT}, we obtain
\begin{align}
&\quad \frac{1}{\sqrt{pq}}(\E[|D_kF|^{2s}])^{1/s} (\E[\vert D_kF \vert^{t}])^{1/t}\notag\\
&=\frac{1}{\sqrt{pq}}\frac{(pq)^{3/2}}{\big(\Var[T]\big)^{3/2}}\Big(\E\Big[\Big(\frac{\sqrt{\Var[T]}}{\sqrt{pq}}D_kF\Big)^{2s}\Big]\Big)^{1/s} \Big(\E\Big[\Big(\frac{\sqrt{\Var[T]}}{\sqrt{pq}}D_kF\Big)^{t}\Big]\Big)^{1/t}\notag\\
&=\frac{pq}{\big(\Var[T]\big)^{3/2}}\mu_{2s}^{1/s}\mu_{t}^{1/t}\,.\label{eq:A4a}
\end{align}
Note that the absolute values are omitted since $D_kF$ is non-negative. Resorting to \eqref{eq:triangleVariance} and \eqref{eq:momentBinom} and using that $\frac 1s+\frac 1t=1-\frac 1r$, we get
\begin{equation}
\label{eq:A4b}
 \frac{pq}{\big(\Var[T]\big)^{3/2}}\mu_{2s}^{1/s}\mu_{t}^{1/t}\asymp\begin{cases}\theta^{-{7\over 2}}\,n^{-3+{\alpha\over 2}}&\text{if }0\leq \alpha\leq\frac 12\\\theta^{-{5\over 2}-{1\over r}}\,n^{-\frac 72+{3\alpha\over 2}+{2\alpha\over r}-{1\over r}}&\text{if }\frac 12<\alpha< 1\,.\end{cases}
\end{equation}
Combining \eqref{eq:rmoments}, \eqref{eq:A4a} and \eqref{eq:A4b}, we obtain that for all even integers $r>2$,
\begin{align}
 A_4=\begin{cases}\mathcal{O}\big(n^{-1+{1\over 2}\alpha}\big)&\text{if }0\leq \alpha\leq\frac 12\\ \mathcal{O}\big(n^{-\frac 32+{3\alpha\over 2}+{2\alpha\over r}-{1\over r}}\big)&\text{if }\frac 12<\alpha< 1\,.\end{cases}\label{eq:A4c}
\end{align}
If $0\leq\alpha\leq\frac 12$, the bound in \eqref{eq:A4c} does not depend on $r$ and is of lower order compared to the bounds in \eqref{eq:A1}--\eqref{eq:A7}.
In the case $\frac 12<\alpha<\frac 23$ the term $A_1$ in \eqref{eq:A1} determines the leading-order asymptotics among the bounds in \eqref{eq:A1}--\eqref{eq:A7} if $r>2$ is chosen in such a way that
\[-\frac 32+\frac 32\alpha+\frac 2r\alpha-\frac 1r\leq -\frac 54+\frac 54\alpha\,,\quad\text{or, equivalently}\,,\quad r\geq\frac{4(2\alpha-1)}{1-\alpha}\,.\]
Namely, we put $r$ as the smallest even integer larger or equal to $\max\big\{2,\frac{4(2\alpha-1)}{1-\alpha}\big\}$ and conclude from \eqref{eq:A4c} that
\begin{align}
A_4=\begin{cases}\mathcal{O}\big(n^{-1+{\alpha\over 2}}\big)&\text{if }0\leq\alpha\leq\frac 12\\ \mathcal{O}\big(n^{-{5\over 4}+{5\alpha\over 4}}\big)&\text{if }\frac 12<\alpha< 1\,.\end{cases}\label{eq:A4}
\end{align}
Moreover, if $\frac 23\leq\alpha<1$, the Kolmogorov distance is dominated by the term $A_6$. 
This concludes the proof.
\end{proof}

After having established Theorem \ref{thm:ERGraph} we turn to the proof of Theorem \ref{thm:Subgraphs}. Recall that in this situation $p\in(0,1)$ is fixed and that $\Gamma$ is a graph with at least one edge. Furthermore, $S$ is the number of copies of $\Gamma$ in $G(n,p)$ and $F:=(S-\E[S])/\sqrt{\Var[S]}$ denotes the normalized sub-graph counting statistic. Let us recall from \cite[Lemma 3.5]{JanLucRuc} that
\begin{equation}\label{eq:VarSubgraphs}
\Var[S] \asymp c(p,\Gamma)\,n^{2v-2}\,,
\end{equation}
where $c(p,\Gamma)\in(0,\infty)$ is a constant only depending on $p$ and $\Gamma$, and where $v=v(\Gamma)$ stands for the number of vertices of $\Gamma$. Finally, we recall that $I$ stands for the set $\{1,\ldots,{n\choose 2}\}$ and put $q:=1-p$.

\begin{proof}[Proof of Theorem \ref{thm:Subgraphs}.]
First, we assume that $n\geq v\geq 4$. Note that for $k\in I$, $S_k^{+}$ and $S_k^{-}$ are the number of copies of $\Gamma$ if edge $e_k$ is present in $G(n,p)$ or not, respectively. Thus, $S_k^{+}-S_k^{-}$ is the number of copies of $\Gamma$ in $G(n,p)$ sharing edge $e_k$. Since there are ${n-2\choose v-2}$ choices for the remaining vertices needed to build such a copy, we have that 
$$D_k F = {\sqrt{pq}\over\sqrt{\Var[S]}}\big(S_k^+-S_k^-\big) = \mathcal{O}(n^{-1})\,,\qquad k\in I\,,$$
where we also used \eqref{eq:VarSubgraphs}. Next, we consider the second-order discrete gradient
$$D_\ell D_k F = {pq\over\sqrt{\Var[S]}}\big((S_k^+)_\ell^+-(S_k^+)_\ell^--(S_k^-)_\ell^++(S_k^-)_\ell^-\big)\,,\qquad k,\ell\in I\,.$$
If $|e_k\cap e_\ell|=0$, $v-4$ further vertices are needed to build a copy of $\Gamma$ containing the edges $e_k$ and $e_\ell$. Since there are ${n-4\choose v-4}$ choices for these vertices and because of \eqref{eq:VarSubgraphs}, one has that
\begin{equation}\label{eq:Sub1}
D_\ell D_k F = \mathcal{O}(n^{-3})\,.
\end{equation}
Similarly, if $|e_k\cap e_\ell|=1$ we find that
\begin{equation}\label{eq:Sub2}
D_\ell D_k F = \mathcal{O}(n^{-2})
\end{equation}
and if $|e_k\cap e_\ell|=2$ we have $k=\ell$ and hence
\begin{equation}\label{eq:Sub3}
D_\ell D_kF = 0\,.
\end{equation}
We can now evaluate the terms arising in Theorem \ref{Folgetheorem}, which we denote by $A_1,\ldots,A_7$. For $A_1$ we have that
$$A_1^2:={15\over 4}\sum_{j,k,\ell\in I}(\E[(D_jF)^2(D_kF)^2])^{1/2}(\E[(D_\ell D_jF)^2(D_\ell D_kF)^2])^{1/2} \,.$$ Using the Cauchy-Schwarz inequality, we see that
\begin{align*}
A_1^2 \leq {15\over 4}\sum_{\ell\in I}\Big(\sum_{k\in I}(\E[(D_kF)^4])^{1/4}(\E[(D_\ell D_kF)^4])^{1/4}\Big)^2
\end{align*}
and a distinction of the cases $|e_k\cap e_\ell|=0$, $|e_k\cap e_\ell|=1$ and $|e_k\cap e_\ell|=2$ yields
$$A_1=\mathcal{O}(n^{-1})$$
by \eqref{eq:Sub1}, \eqref{eq:Sub2} and \eqref{eq:Sub3}. Similar considerations with $r=2$ and $s=t=4$ lead to
\begin{alignat*}{3}
& A_2 = \mathcal{O}(n^{-2})\,,\qquad && A_3 = \mathcal{O}(n^{-1})\,,\qquad && A_4 = \mathcal{O}(n^{-1})\,,\\
& A_5 = \mathcal{O}(n^{-1})\,,\qquad && A_6 = \mathcal{O}(n^{-3/2})\,,\qquad && A_7 = \mathcal{O}(n^{-5/2})
\end{alignat*}
and hence to $d_K(F,N)=\mathcal{O}(n^{-1})$. 

The case that $\Gamma$ has exactly two vertices is covered by the classical Berry-Esseen theorem for a binomial distribution with parameters $n\choose 2$ and $p$. If $\Gamma$ has exactly three vertices, then $\Gamma$ is either the complete graph on $3$ vertices (as already covered by Theorem \ref{thm:ERGraph}) or a graph with $1$ or $2$ edges on $3$ vertices, respectively. In these cases, instead of \eqref{eq:Sub1} one has that $D_\ell D_k F=0$ if $|e_k\cap e_\ell|=0$ and one obtains that $d_K(F,N)=\mathcal{O}(n^{-1})$. This completes the proof.
\end{proof}

Finally in this section, we turn to the proof of Theorem \ref{thm:degrees} for which we use the same set-up as in the proof of Theorem \ref{thm:ERGraph}. In particular, we denote by $I$ the set $\{1,\ldots,{n\choose 2}\}$ and recall that $p=\theta n^{-\alpha}$ with suitable $\alpha\in\R$ and $\theta\in(0,n^{\alpha})$ such that $\theta\asymp 1$. We also put $q:=1-p$. For $d\in\{0,1,2,\ldots\}$ we denote by $V_{n,d}$ the number of vertices of $G(n,p)$ with degree $d$ and put $G_{n,d}:=(V_{n,d}-\E[V_{n,d}])/\sqrt{\Var[V_{n,d}]}$. Let us recall from Chapter 6.3 in \cite{JanLucRuc} that 
\begin{equation}\label{eq:Var0}
\Var[V_{n,0}]\asymp 2\theta n^2 p = 2\theta n^{2-\alpha}\,,\qquad \alpha\in[1,2)\,,
\end{equation}
and for $d\in\N$,
\begin{equation}\label{eq:Var1}
\Var[V_{n,d}] \asymp c(d,\theta)\,n^{d+1}p^d = c(d,\theta)\,n^{d+1-\alpha d}\,,\qquad \alpha\in[1,1+1/d)\,,
\end{equation}
with a constant $c(d,\theta)\in(0,\infty)$ only depending on $d$ and on $\theta$. From Theorem 8 in \cite{BarKarRuc} it is known that a central limit theorem for $G_{n,0}$ holds if and only if $n^2p\to\infty$ and $np-\log n\to-\infty$, as $n\to\infty$. In our case that $p=\theta n^{-\alpha}$ this is equivalent to $\alpha\in[1,2)$. Moreover, \cite[Theorem 6.36]{JanLucRuc} says that for $d\in\N$, $G_{n,d}$ satisfies a central limit theorem if and only if $n^{d+1}p^d\to\infty$ and $np-\log n-d\log\log n\to-\infty$, as $n\to\infty$. Again, in our case this is equivalent to $\alpha\in[1,1+1/d)$, whence the conditions on $\alpha$ in \eqref{eq:Var0} and \eqref{eq:Var1}.

\begin{proof}[Proof of Theorem \ref{thm:degrees}]
At first, we notice that adding or removing an edge from $G(n,p)$ can change the number of vertices with degree equal to $d\in\{0,1,2,\ldots\}$ by at most $2$. This implies that
\begin{align}\label{eq:VertexDkBound1}
|D_kG_{n,d}|\leq {2\sqrt{pq}\over\sqrt{\Var[V_{n,d}]}}\,,\qquad k\in I\,.
\end{align}
Next, we observe that $(pq)^{-1}|D_kD_\ell  V_{n,d}|\in\{0,1,2\}$ for all $k,\ell\in I$. We also have that $D_kD_\ell  V_{n,d}$ and hence $D_kD_\ell  G_{n,d}$ is zero whenever the two corresponding edges $e_k$ and $e_\ell $ are identical or do not share a common vertex. Resorting to the definition of the random variable $G_{n,d}$, we thus conclude that
\begin{align}\label{eq:VertexDkBound2}
|D_kD_\ell  G_{n,d}|\leq {2pq\over\sqrt{\Var[V_{n,d}]}}\1_{\{|e_k\cap e_\ell |=1\}}\,,\qquad k,\ell\in I\,.
\end{align}

We can now evaluate the bound in Theorem \ref{Folgetheorem}. We start with the case $d=0$. Since the computations are almost identical for each of the terms there, we restrict to the first term $A_1$, which is given by
$$A_1:=\Big({15\over 4}\sum_{j,k,\ell\in I}(\E[(D_jG_n)^2(D_kG_n)^2])^{1/2}(\E[(D_\ell D_jG_n)^2(D_\ell D_kG_n)^2])^{1/2} \Big)^{1\over 2}\,.$$
Using \eqref{eq:VertexDkBound1} and \eqref{eq:VertexDkBound2}, we see that
\begin{align}\label{eq:A1Vertex}
A_1^2 \leq {60(pq)^3\over(\Var[V_{n,d}])^2}\sum_{j,k,\ell\in I}\1_{\{|e_j\cap e_\ell|=1\}}\1_{\{|e_k\cap e_\ell|=1\}} =  {60(pq)^3\over(\Var[V_{n,d}])^2}{n\choose 2}(2(n-2))^2\,.
\end{align}
Now, we use that $p=\theta n^{-\alpha}$ as well as the variance asymptotics at \eqref{eq:Var0}. This allows us to conclude that $A_1=\mathcal{O}(n^{-\alpha/2})$. Denoting the other terms arising in Theorem \ref{Folgetheorem} by $A_2,\ldots,A_7$, we conclude by similar computations and by taking $r=2$, $s=t=4$ that
\begin{alignat*}{3}
& A_2 =\mathcal{O}(n^{-\alpha/2})\,,\qquad && A_3 = \mathcal{O}(n^{-1+\alpha/2})\,,\qquad && A_4 = \mathcal{O}(n^{-1+\alpha/2})\,,\\
& A_5 =\mathcal{O}(n^{-1+\alpha/2})\,,\qquad && A_6= \mathcal{O}(n^{-1/2})\,,\qquad && A_7= \mathcal{O}(n^{-1/2})\,.
\end{alignat*}
Thus, $d_K(G_{n,0},N)=\mathcal{O}(n^{-1+\alpha/2})$. Turning to the case $d\in\N$ we start again with the term $A_1$ and obtain by using \eqref{eq:A1Vertex} and \eqref{eq:Var1} that $A_1=\mathcal{O}(n^{1-d-3\alpha/2+\alpha d})$. Moreover, one sees for the other terms $A_2,\ldots,A_7$ in Theorem \ref{Folgetheorem} that
\begin{alignat*}{3}
& A_2 =\mathcal{O}(n^{1-d-3\alpha/2+\alpha d})\,,\qquad && A_3 = \mathcal{O}(n^{1/2-3d/2-\alpha+3\alpha d/2})\,,\qquad && A_4 = \mathcal{O}(n^{1/2-3d/2-\alpha+3\alpha d/2})\,,\\
& A_5 =\mathcal{O}(n^{-d-\alpha/2+\alpha d})\,,\qquad && A_6= \mathcal{O}(n^{1/2-d-\alpha+\alpha d})\,,\qquad && A_7= \mathcal{O}(n^{1/2-d-\alpha+\alpha d})\,.
\end{alignat*}
Thus, for $d\in\N$, $d_K(G_{n,d},N)=\mathcal{O}(n^{1/2-3d/2-\alpha+3\alpha d/2})$, provided that $\alpha\in[1,(3d-1)/(3d-2))$. This completes the proof.
\end{proof}

\section{Application to percolation on trees and proof of Theorem \ref{thm:PercolationTree}}\label{sec:Percolation}

Let us recall some notation and embed the objects into the framework of Sections \ref{sec:Preliminaries} and \ref{sec:AbstractBerryEsseen}. We denote by $\sT$ an infinite rooted tree such that each vertex has degree bounded by $D+1$ with $D\in\N$. By $\sT_n$, $n\in\N$, we indicate the finite sub-tree of $\sT$ consisting of all vertices with graph-distance at most $n$ from the root. We now embed $\sT$ into the Euclidean plane by the following procedure, which is illustrated in Figure \ref{fig:tree2}. The root is mapped to the point with coordinates $(1,1)$ and the vertices adjacent to it are mapped to the points with coordinates $(1,2),\ldots,(N(1),2)$ with $N(1)\leq D$ in an arbitrary order. Next, the vertices adjacent to these are mapped onto $(1,3),\ldots,(N(2),3)$, where (from left to right) the first points in this list are adjacent to $(1,2)$, the next points to $(2,2)$, etc. Continuing this way, the vertices with graph-distance $n$ to the root are mapped onto $(1,n+1),\ldots,(N(n),n+1)$ with $N(n)\leq N(n-1)D$ and the infinite tree $\sT$ is embedded into the upper right quadrant of the Euclidean plane. A vertex of the embedded tree with coordinates $(i,k)$ for $k\in\N$ and $i\in\{1,\ldots,N(k)\}$ receives the label $1+N(1)+\ldots+N(k-1)+i$ and an edge of the embedded tree whose adjacent vertices have coordinates $(i,k)$ and $(j,k-1)$ for $k\in\{2,3,\ldots\}$, $i\in\{1,\ldots,N(k)\}$ and $j\in\{1,\ldots,N(k-1)\}$ finally receives the label of its endpoint minus one, i.e. $N(1)+\ldots+N(k-1)+i$, see Figure \ref{fig:tree2}. This numbering of vertices also corresponds to that in Figure \ref{fig:tree}.

\begin{figure}[t]
\includegraphics[width=0.5\columnwidth]{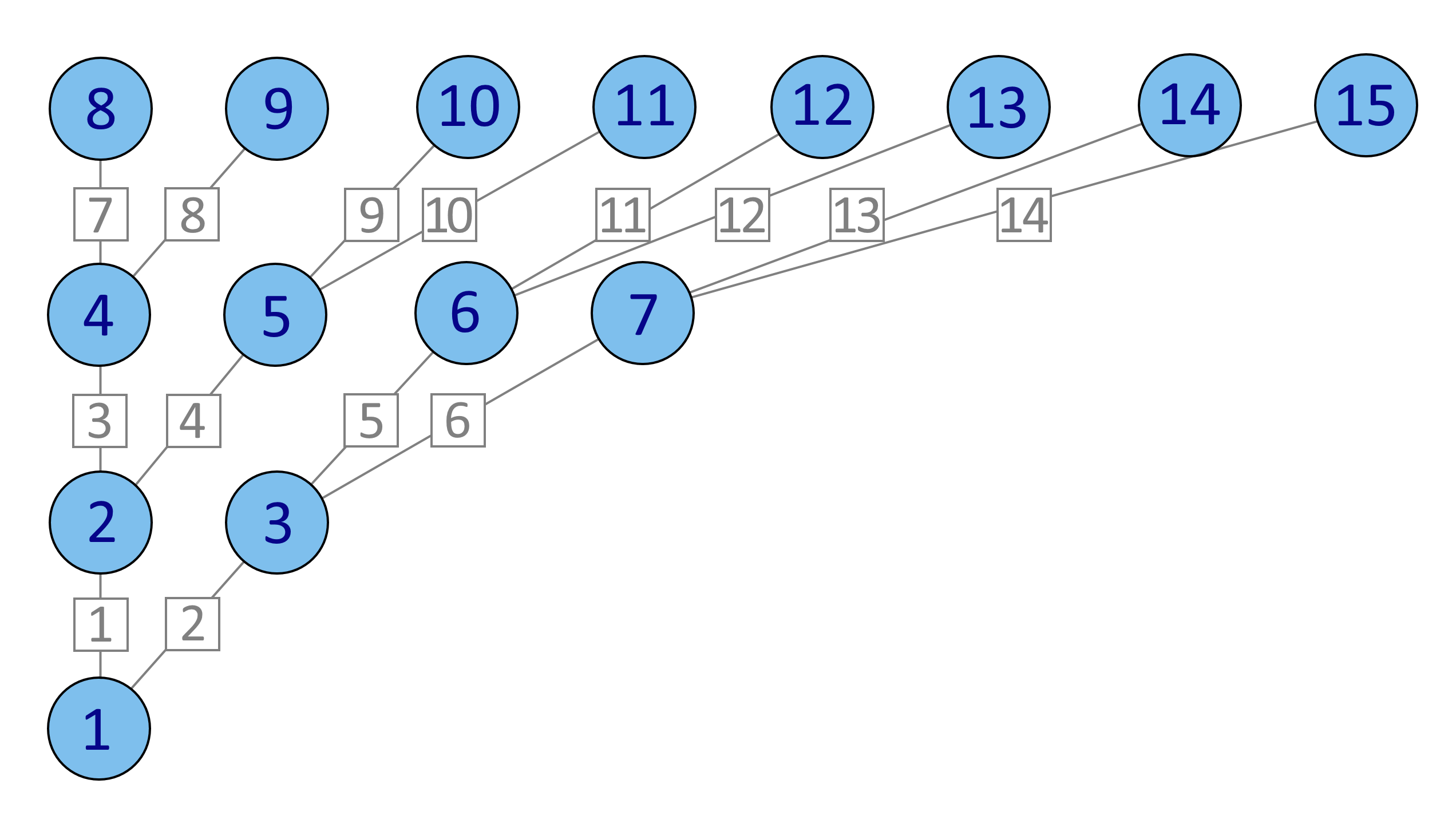}
\caption{Embedding in the plane of $\sT_3$ of a $2$-regular tree $\sT$.}
\label{fig:tree2}
\end{figure}

This construction puts us in the position to interpret our percolation problem on $\sT$ in terms of the framework of Theorem \ref{AbstractBound}. Namely, for fixed $p\in(0,1)$ let $(X_k)_{k\in\N}$ be a sequence of independent Rademacher random variables with $P(X_k=+1)=p$ and $P(X_k=-1)=1-p$. For each $k\in\N$, assign the random variable $X_k$ to the uniquely determined edge $e_k$ of the embedded tree with label $k$. The random graph $\sT(p)$ consists of all edges $e_k$ of the embedded tree with label $X_k=+1$ together with their two adjacent vertices. Thus, $\sT(p)$ is described by the Rademacher sequence $(X_k)_{k\in\N}$ and its restriction $\sT_n(p)$ to $\sT_n$ is described by a finite sub-sequence of $(X_k)_{k\in\N}$.

For $n\in\N$, we denote by $C_n(p)$ the number of connected components of the random graph $\sT_n(p)$, where, as already discussed in the introduction, by a connected component we understand a maximal connected sub-graph with at least one edge. By $H_n(p):=(C_n(p)-\E[C_n(p)])/\sqrt{\Var[C_n(p)]}$ we denote the normalized version of $C_n(p)$ and notice that $C_n(p)$ is a Rademacher functional.
 
\begin{proof}[Proof of Theorem \ref{thm:PercolationTree}]
We start by investigating the first- and second-order discrete gradient applied to $H_n(p)$. By definition, we have that
$$D_kH_n(p)={\sqrt{pq}\over\sqrt{\Var[C_n(p)]}}D_kC_n(p)={\sqrt{pq}\over\sqrt{\Var[C_n(p)]}}\big((C_n(p))_k^+-(C_n(p))_k^-\big)\,,$$
where $k\in\{1,\ldots,1+N(1)+\ldots+N(n)\}$.  Note that $D_kC_n(p)$ is a local quantity since it depends only on the edges adjacent to $k$. Adding or removing the edge with label $k$ can change the number of connected components by at most $1$. Therefore, we have that
\begin{equation}
\label{eq:est1}
|D_kH_n(p)|\leq {\sqrt{pq}\over\sqrt{\Var[C_n(p)]}}
\end{equation}
for all $k\in\{1,\ldots,1+N(1)+\ldots+N(n)\}$. Next, we consider for $k,j\in\{1,\ldots,1+N(1)+\ldots+N(n)\}$ the second-order discrete gradient
$$D_kD_jH_n(p) = {pq\over\sqrt{\Var[C_n(p)]}}\big(((C_n(p))_j^+)_k^+-((C_n(p))_j^+)_k^--\big(((C_n(p))_j^-)_k^++((C_n(p))_j^-)_k^-\big)\big)\,.$$
For most choices of $j$ and $k$, $D_kD_jH_n(p)$ is zero. A non-zero contribution only arises if the edges $e_j$ and $e_k$ with labels $j$ and $k$, respectively, share precisely one common vertex. We indicate this situation by $|e_j\cap e_k|=1$ and write $|e_j\cap e_k|\in\{0,2\}$ otherwise.  Thus, we can use the triangle inequality and the estimate \eqref{eq:est1} to conclude that
\begin{align}
\label{eq:est2}
|D_jD_kH_n(p)|\;\begin{cases} = 0 &\text{if }|e_j\cap e_k|\in\{0,2\}\\ \leq {2\,pq\over\sqrt{\Var[C_n(p)]}} &\text{if }|e_j\cap e_k|=1\,.\end{cases}
\end{align}

We use a lower bound for the variance of $C_n(p)$, which can be found in \cite[Identity (2.3)]{SugimineTakei} in case of a $D$-regular tree, but the proof is easily seen to carry over to our situation. More precisely, there exists a constant $c(p)>0$ only depending on $p$ such that 
\begin{equation}
\label{eq:vlb}
\Var[C_n(p)]\geq c(p)\,|\sT_n|\,.
\end{equation}
Estimating the terms in Theorem \ref{Folgetheorem} with $r=2$ and $s=t=4$ there by means of \eqref{eq:est1}--\eqref{eq:vlb} yields (after a straight forward computation similar to the one in the proof of Theorem \ref{thm:degrees}) that
\[d_K\left(H_n(p),N\right)=\mathcal{O}(|\sT_n|^{-1/2})\,.\]
In case of a $D$-regular tree, we have that $|\sT_n|=D+\ldots+D^n=(D^{n+1}-1)/(D-1)-1$, if $D\geq 2$, and $|\sT_n|=n$, if $D=1$. Thus, $|\sT_n|$ behaves like $D^n$, if $D\geq 2$, and like $n$, if $D=1$, as $n\to\infty$. This completes the proof.
\end{proof}

\subsection*{Acknowledgement}
We are greatly indebted to the referee for a careful reading and for the many helpful hints and suggestions. We would also like to thank Peter Eichelsbacher for insightful discussions on the alternative approaches to Theorem \ref{thm:Subgraphs} and for the remarks by Larry Goldstein he communicated to us.

The authors were supported by the German Research Foundation (DFG) via SFB-TR 12.

\bibliography{RademacherRevision-1}

\end{document}